\documentclass{amsart}
\usepackage{amsmath,amssymb, amscd, stmaryrd,bm}
\usepackage{fullpage}
\usepackage{enumerate}
\usepackage{cite}
\usepackage{framed}
\usepackage[dvips,dvipdf]{graphicx}
\usepackage[usenames,dvipsnames]{color}
\usepackage{color}
\usepackage{tabularx}
\usepackage{array}
\usepackage{multirow}
\usepackage{changes}
\usepackage{subfig}
\usepackage[T1]{fontenc}

\newtheorem{theorem}{Theorem}[section]
\newtheorem{proposition}[theorem]{Proposition}

\theoremstyle{definition}

\theoremstyle{definition}

\theoremstyle{remark}
\newtheorem{remark}[theorem]{Remark}
\newtheorem{example}[theorem]{Example}

\graphicspath{{Figures/}} 

\newcommand{\enorm}[1]{\ensuremath{|\!|\!|#1|\!|\!|}}
\newcommand{\norm}[1]{\ensuremath{\left\| #1\right\| }}
\newcommand{\mb}[1]{\ensuremath{\mathbf{#1}}}

\newcommand{\NN}{\mathbb N}

\newcommand{\QQ}{\mathbb Q}
\newcommand{\RR}{\mathbb R}

\newcommand{\spec}{\mathrm{Spec}}

\newcommand{\cA}{\mathcal{A}}
\newcommand{\cE}{\mathcal E}
\newcommand{\cF}{\mathcal F}
\newcommand{\cH}{\mathcal H}

\newcommand{\cT}{\mathcal T}

\author[S. Giani]{Stefano Giani} \address{Durham University, School of Engineering and Computing Sciences, South Road, Durham DH1 3LE, United Kingdom}
\email{stefano.giani@durham.ac.uk}

\author[L. Grubi\v{s}i\'{c}]{Luka Grubi\v{s}i\'{c}} \address{University of Zagreb, Department of Mathematics, Bijeni\v{c}ka 30, 10000 Zagreb, Croatia}
\email{luka@math.hr}

\author[H. Hakula]{Harri Hakula} \address{Department of Mathematics and Systems Analysis, Aalto University, Finland}
\email{harri.hakula@aalto.fi}

\author[J. Ovall]{Jeffrey S. Ovall} \address{Jeffrey S. Ovall,
  Fariborz Maseeh Department of Mathematics and Statistics, Portland
  State University, Portland, OR 97201}
\email{jovall@pdx.edu}

\begin{document}
\title[Error Estimates for Eigenvalue Problems]{A Posteriori Error Estimates for Elliptic Eigenvalue Problems Using
  Auxiliary Subspace Techniques} \date{\today}

\begin{abstract}
  We propose an a posteriori error estimator for
  high-order $p$- or $hp$-finite element discretizations of
  selfadjoint linear elliptic eigenvalue problems that is appropriate
  for estimating the error in the approximation of an eigenvalue cluster
  and the corresponding invariant subspace.  The estimator is based on
  the computation of approximate error functions in a space that
  complements the one in which the approximate eigenvectors were
  computed.  These error functions are used to construct estimates of
  collective measures of error, such as the Hausdorff distance between
  the true and approximate clusters of eigenvalues, and the subspace gap
  between the corresponding true and approximate invariant subspaces.
  Numerical experiments demonstrate the practical effectivity of the
  approach.
\end{abstract}

\maketitle

\section{Introduction}\label{Intro}

This paper concerns the a posteriori estimation of error in high-order
($p$ or $hp$) finite element approximations of eigenvalues and
invariant subspaces for variational eigenvalue problems of the form:
Find $(\lambda,\psi)\in \RR\times \cH$, $\psi\neq 0$, satisfying
\begin{align}\label{VarEig}
\underbrace{\int_\Omega A\nabla\psi\cdot\nabla v+b\psi v\,dx}_{B(\psi,v)}=\lambda\underbrace{\int_\Omega
  \psi v\,dx}_{(\psi,v)}\mbox{ for all }v\in\cH~,
\end{align}
where $\Omega\subset\RR^d$ is open and bounded, and $\cH\subset
H^1(\Omega)$ incorporates homogeneous Dirichlet, Neumann, or mixed
Dirichlet/Neumann boundary conditions. Standard assumptions on the
coefficients $A\in [L^\infty(\Omega)]^{d\times d}$ and $b\in
L^\infty(\Omega)$ ensure that $B$ is an inner-product on $\cH$, whose
induced ``energy'' norm, $\enorm{v}=\sqrt{B(v,v)}$, is equivalent to
the standard norm on $H^1(\Omega)$, $\|v\|_1$.  We also use $\|v\|_0$
to denote the standard norm on $L^2(\Omega)$.

We will compute a collection of approximate eigenvalues and
eigenvectors using either $p$ or $hp$ finite element discretizations
(see Section~\ref{HPDiscretization} for details).  Let $V\subset\cH$
denote such a finite element space.  The corresponding discrete
version of~\eqref{VarEig} is: Find $(\hat\lambda,\hat\psi)\in \RR\times V$, $\hat\psi\neq 0$ satisfying
\begin{align}\label{DiscVarEig}
B(\hat\psi,v)=\hat\lambda(\hat\psi,v)\mbox{ for all }v\in V~.
\end{align}
For convenience, we state a few well-known results concerning the
solutions of~\eqref{VarEig} and~\eqref{DiscVarEig}.  
\begin{enumerate}
\item The problem~\eqref{VarEig} admits countably many solutions
  $\{(\lambda_n,\psi_n):\,n\in\NN\}$, such that
\begin{enumerate}
\item $0< \lambda_1 <\lambda_2\leq \lambda_3\leq \cdots$, and
  $\{\lambda_n\}$ has no finite accumulation points;
\item $\{\psi_n\}$ is an orthonormal Hilbert basis of $L^2(\Omega)$.
\end{enumerate}
\item The problem~\eqref{DiscVarEig} admits $N=\dim(V)$ solutions
  $\{(\hat\lambda_n,\hat\psi_n):\,1\leq n\leq N\}$, such that
\begin{enumerate}
\item $0< \hat\lambda_1 \leq\hat\lambda_2\leq \cdots\leq\hat\lambda_N$;
\item $\{\hat\psi_n\}$ is an $L^2(\Omega)$-orthonormal basis of $V$.
\end{enumerate}
\item $\lambda_n\leq \hat\lambda_n$ for $1\leq n\leq N$.
\end{enumerate}

One feature of eigenvalue problems that complicates the estimation of
error is the possibility of repeated or tightly-clustered eigenvalues,
which arise very naturally in domains with symmetries or
near-symmetries, and will heavily feature in our numerical
experiments.  When such eigenvalues are to be approximated in
practice, it may make little sense to try to determine whether
computed eigenvalue approximations that are very close to each other
are all approximating the same (repeated) eigenvalue, or approximating
eigenvalues that just happen to be very close to each other.  In this
case, it is best to estimate eigenvalue error and associated invariant
subspace error in a ``collective sense'', as described in
Section~\ref{Theory}.  Let us briefly outline approaches to
``collective'' eigenvalue estimates in the literature. First there is
an approach using majorization inequalities championed by A. Knyazev
in a series of papers, see for instance \cite{Knyazev2009} and the
references therein.  Majorization inequalities yield optimal estimates
for clusters of eigenvalues on the extreme portions of the spectrum
and Knyazev's approach is focused on a priori estimates. See also the
notion of cluster robustness from \cite{Ovtchinnikov2006} in the
context of a posteriori estimates. These estimates are optimal for the
eigenvalues on the boundary of the spectrum and involve only
``diagonal part'' or ``trace'' of the subspace residual, see Section
\ref{Theory} for more details.  An alternative approach involves the
use of Hausdorff distance between the ``matched'' groups of
eigenvalues and their approximants, as well as a measure of the
subspace gap between the true invariant subspace and its
approximation, see~\cite{Boffi2017}.  As with~\cite{Boffi2017}, we are
principally interested in a posteriori estimates of error measured in
Hausdorff distance (for eigenvalues) and subspace gap (for
eigenvectors), but both our analysis and the practical realization of
the estimators take on a very different form. 

As is the case with solutions of source
problems (boundary value problems), eigenvectors can have
singularities due to domain geometry and/or discontinuities in the
differential operator or boundary conditions, and the types and
severity of singularities that can occur are
well-understood~\cite{Grisvard1992,Kozlov1997,Wigley1964}.  Unlike
source problems, where the strongest singular behavior that can be
present is typically seen in practice, with eigenvalue problems, the
regularity of eigenvectors varies (dramatically) depending on where
you are in the spectrum, as illustrated in the following example.
We consider this example in detail, first focusing on an eigenvalue
cluster of mixed regularity and illustrating the notion of
``mixing of eigenmodes'' at different levels of discretization, and later
revisiting it to demonstrate the performance (effectivity) of our a posteriori error
estimates---the eigenvalues and vectors are known, so the errors and
error estimates can be directly compared. 

\begin{figure}
	\centering
	\subfloat[{$p=4$: $\hat\psi_{52}\approx\psi_{41}$.}]
	{\includegraphics[width=0.24\textwidth]{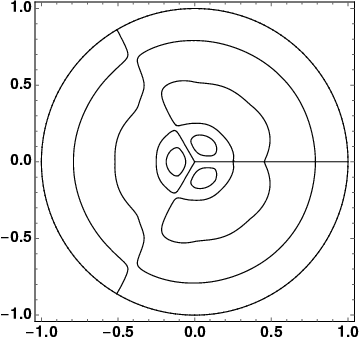}}
	\subfloat[{$p=5$: $\hat\psi_{52}\approx\psi_{55}$.}]
	{\includegraphics[width=0.24\textwidth]{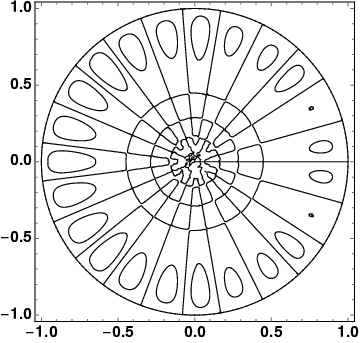}}
	\subfloat[{$p=6$: $\hat\psi_{52}\approx\psi_{55}$.}]
	{\includegraphics[width=0.24\textwidth]{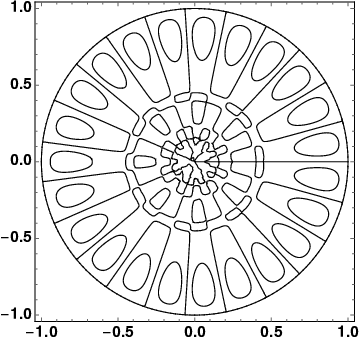}}
	\subfloat[{$p=7$: $\hat\psi_{52}\approx\psi_{52}$.}]
	{\includegraphics[width=0.24\textwidth]{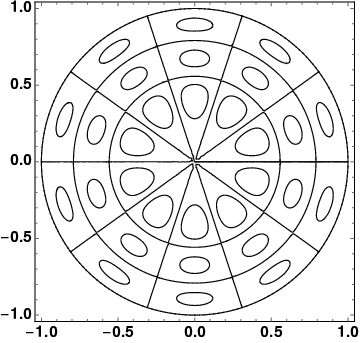}}\\
	\subfloat[{$p=4$: $\hat\psi_{53}\approx\psi_{59}$.}]
	{\includegraphics[width=0.24\textwidth]{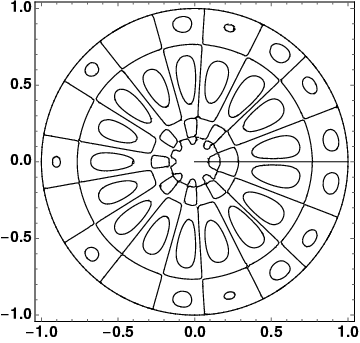}}
	\subfloat[{$p=5$: $\hat\psi_{53}\approx\psi_{52}$.}]
	{\includegraphics[width=0.24\textwidth]{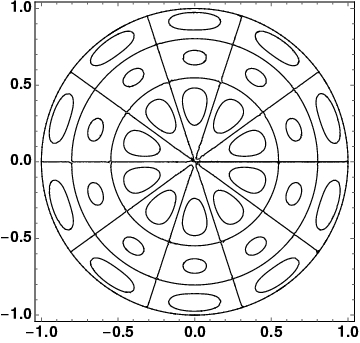}}
	\subfloat[{$p=6$: $\hat\psi_{53}\approx\psi_{52}$.}]
	{\includegraphics[width=0.24\textwidth]{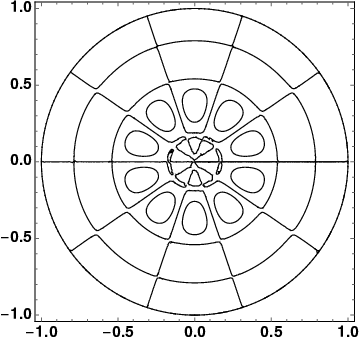}}
	\subfloat[{$p=7$: $\hat\psi_{53}\approx\psi_{53}$.}]
	{\includegraphics[width=0.24\textwidth]{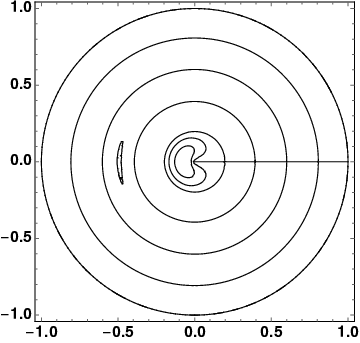}}
	\caption{Contour plots of computed eigenvectors $\hat\psi_{52}$ (top row) and
          $\hat\psi_{53}$ (bottom row) on a sequence of increasingly fine discretizations.
	}\label{fig:SlitDiskConvergence52_53}
\end{figure}

\begin{example}[Slit Disk]\label{SlitDiskExample}
Let $\Omega\subset\RR^2$ be the unit disk with the positive $x$-axis
removed, and consider the
Laplace eigenvalue problem:
\begin{align*}
-\Delta\psi=\lambda\psi\mbox{ in }\Omega\quad,\quad \psi=0\mbox{ on }\partial\Omega
\end{align*}
The eigenvalues and vectors are known explicitly (cf.~\cite{Kuttler1984}), and are
doubly-indexed for $m,n\in\NN$ by
\begin{align}\label{eq:exactmodes}
\psi_{m,n}=J_{n/2}(j_{m,n}r)\sin(n\theta/2)\quad,\quad \lambda_{m,n}=j_{m,n}^2~,
\end{align}
where $J_{n/2}$ is the first-kind Bessel function of order $n/2$ and
$j_{m,n}$ is the $m$th positive root of $J_{n/2}$; $r\in[0,1]$ and $\theta\in[0,2\pi]$
are the usual polar coordinates.  
Since $J_{1/2}(z)=\sqrt{2/(\pi
  z)}\,\sin z$, we see that $\lambda_{m,1}=(m\pi)^2$, and
$\psi_{m,1}\in H^{3/2-\epsilon}(\Omega)$ only for $\epsilon>0$.

It is well-known that, when $\nu\in\QQ$ and $\ell\in\NN$, then $J_{\nu}$
and $J_{\nu+\ell}$ have no common positive roots
(cf~\cite[pp. 484-485]{Watson1995}), and that the positive roots of
Bessel functions are simple.  
It follows from the first of these assertions that
$J_{n/2}$ and $J_{n'/2}$ have no common positive roots when $n$ and
$n'$ have the same parity, but it does not rule out that they may have
common positive roots when  $n$ and
$n'$ do not have the same parity.  We have not determined whether or not
all eigenvalues in this example are simple, but we have verified that
at least the first 100 are, which will be sufficient for our purposes.
If the eigenvalues are ordered in an increasing sequence as described
above, this induces a natural mapping $(m,n)\mapsto k$ from index
pairs to absolute indices.  For $k\leq 100$ we know that this map is
invertible, with $52\mapsto (3,10)$ and $53\mapsto (5,1)$, for example.
Contour plots of $\psi_{3,10}=\psi_{52}$ and $\psi_{5,1}=\psi_{53}$
are given, together with their corresponding eigenvalues, in
Figure~\ref{fig:SlitDiskEig52_53}.  This illustrates that eigenmodes
associated with eigenvalues that are relatively close to each other
can have very different regularities; $\psi_{53}\in
H^{3/2-\epsilon}(\Omega)$ only for $\epsilon>0$, but $\psi_{52}\in
H^{\ell}(\Omega)$ for all $\ell$.
\begin{figure}
	\centering
	\subfloat[Eigenvector $\psi_{52}$.]
	{\label{fig:SlitDisk52}\includegraphics[width=0.45\textwidth]{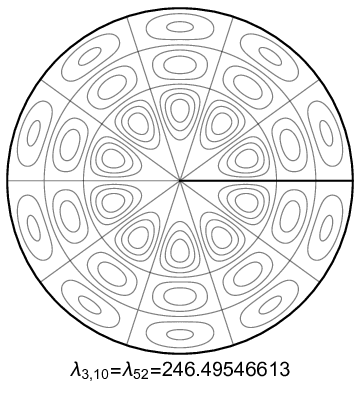}}\quad
	\subfloat[Eigenvector $\psi_{53}$.]
	{\label{fig:SlitDisk53}\includegraphics[width=0.45\textwidth]{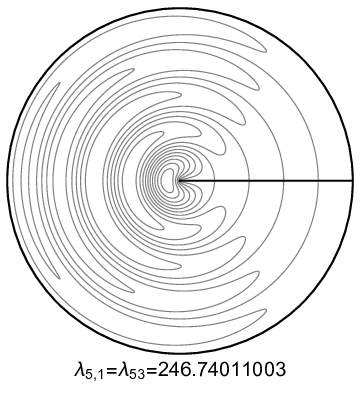}}\\
	\subfloat[Eigenvector error $e_{52}$ and approximation $\varepsilon_{52}$.]
	{
	\includegraphics[width=0.45\textwidth]{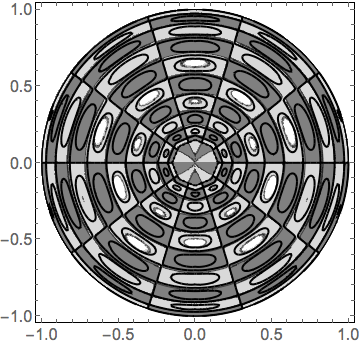}
	}\quad
	\subfloat[Eigenvector error $e_{53}$ and approximation $\varepsilon_{53}$.]
	{
	\includegraphics[width=0.45\textwidth]{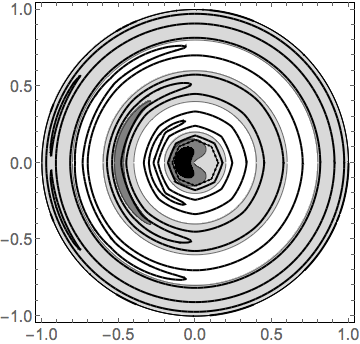}
	}
	\caption{Two consecutive eigenvectors $\psi_j$ for the slit
          disk, together with finite element errors
          $e_j=\psi_j-\hat\psi_j$ and their approximations
          $\varepsilon_j\approx e_j$.  The eigenfunction error $e_j$ is
          given as a greyscale contour plot, and thick black contour
          lines of its approximation
          $\varepsilon_j$ are overlaid.}\label{fig:SlitDiskEig52_53}
\end{figure}

We also use this example to illustrate a \textit{mixing of modes} that
may occur in eigenvalue/vector approximations.  In
Figure~\ref{fig:SlitDiskConvergence52_53} we show contour plots of the
computed eigenvectors $\hat\psi_{52}$ and $\hat\psi_{53}$ corresponding
to $\hat\lambda_{52}$ and $\hat\lambda_{53}$, for a sequence of
increasingly fine discretizations that will be described in
Section~\ref{HPDiscretization}.  The computed eigenvectors are then
identified with the true eigenvectors \textit{they most closely
  resemble}, based on analysis of their behavior (e.g. sign changes)
in both the angular and radial directions.  We describe this procedure
in greater detail later.  We note that $\psi_{53}=\psi_{5,1}$ is approximated by
$\hat\psi_{53}$ only on the finest of these discretizations, whereas
$\psi_{52}=\psi_{3,10}$ is approximated by $\hat\psi_{53}$ on two of
the discretizations, and only moves into its proper position on the
finest discretization---compare with
Figure~\ref{fig:SlitDiskEig52_53}.
In its progression toward approximating $\psi_{52}$, $\hat\psi_{52}$
approximates $\psi_{41}=\psi_{4,3}$ on the coarsest of the
discretizations, and $\psi_{55}=\psi_{1,23}$ on the next two
discretizations.  Similarly, $\hat\psi_{53}$ approximates
$\psi_{59}=\psi_{2,17}$ on the coarsest discretization, and $\psi_{52}$
on the next two discretizations.  The computed approximations
$\hat\lambda_{52}$ and $\hat\lambda_{53}$ both decrease monotonically
toward their respective values $\lambda_{52}$ and $\lambda_{53}$ as
the discretizations are enriched, as they should, with
with $\hat\lambda_{52}=317.923$ and $\hat\lambda_{53}=318.275$ on the
coarsest of the discretizations ($p=4$), and $\hat\lambda_{52}=247.941$ and $\hat\lambda_{53}=250.782$ on the
finest of the discretizations ($p=7$) used for Figure~\ref{fig:SlitDiskConvergence52_53}.

Having identified how fine our discretizations must be in order to
properly identify $\hat\psi_{52}$ and $\hat\psi_{53}$ with $\psi_{52}$
and $\psi_{53}$, we highlight a feature of the error estimation
technique that we propose.  Our approach to error estimation in the
eigenvalue context is based on related work for source
problems~\cite{HNO}, in that eigenvector errors are approximated
\textit{as functions} in an auxiliary space that, in a practical
sense, complements the finite element space in which the eigenvectors
are approximated.  For related work in the context of low-order finite element eigenvalue/vector 
approximations, we refer to~\cite{Grubisic2009,Bank2013}.
Appropriate norms of such approximate error
functions provide the basis for estimating eigenvalue and invariant
subspace errors.  Because we compute approximate eigenvector error
functions, we can provide qualitative, as well as quantitative
estimates of error.  To illustrate this point, we compute approximate
eigenvectors $\hat\psi_{52},\hat\psi_{53}$ in suitable finite element
spaces, and provide contour plots of the errors
$e_j=\psi_{j}-\hat\psi_{j}$ and approximate errors
$\varepsilon_j\approx e_j$, also in Figure~\ref{fig:SlitDiskEig52_53}.
The functions $\psi_{j}$ and $\hat\psi_{j}$ have been normalized so
that $\|\psi_{j}\|_0=\|\hat\psi_{j}\|_0=1$ and $\hat\psi_{j}$ is a
better approximation of $\psi_{j}$ than is $-\hat\psi_{j}$.  
The mesh used for these computations, shown in
Figure~\ref{fig:SlitDisk}, resolves these modes close to the origin with
errors that are an order of magnitude smaller than those a bit farther
away.  Because of this, for visual clarity we have omitted the
contours of $\varepsilon_j$ in the central portion of Figures~\ref{fig:SlitDiskEig52_53}(c)-(d).
\end{example}

The rest of the paper is organized as follows.  In
Section~\ref{Theory} we present general results concerning estimation
of error that is suitable for clusters of eigenvalues and their
corresponding invariant subspaces.  In Section~\ref{HPDiscretization},
we describe the $p$- and $hp$-finite element spaces that are used in this
work, the technique we have used to identify computed eigenmodes with
true eigenmodes when the latter are known (as was done in
Example~\ref{SlitDiskExample}), and our approach for a posteriori
error estimation in this context.  We provide a detailed case study in
Section~\ref{sec:Exper} of examples having many clustered eigenvalues
throughout the spectrum, which were constructed taking a pair of
\textit{isospectral drums} and connecting them in various ways with
narrow bridges.  We focus on 2D problems in
Sections~\ref{HPDiscretization} and~\ref{sec:Exper}, but we
emphasize that the theoretical development in Section~\ref{Theory}
is not dimension-dependent.
\section{Theoretical Results}\label{Theory}

It will be convenient for the development of the error estimates to
express~\eqref{VarEig} in terms of operators.  The bilinear form
defines an operator $\cA$ by a representation theorem of
Friedrichs~\cite{Friedrichs1934} (see also~\cite[Chapter 6, Theorem
2.1]{KatoBook}), such that $(\cA v,w)=B(v,w)$ for all
$v\in\mathrm{Dom}(\cA)\subset\cH$ and $w\in \cH$, and we write
$ \cA v=-\nabla\cdot A\nabla v+bv $.  The operator $\cA$ is
self-adjoint and positive definite, and is typically viewed as an
unbounded operator on $L^2(\Omega)$.  The variational eigenvalue
problem~\eqref{VarEig} is equivalent to the operator eigenvalue
problem: Find $(\lambda,\psi)\in\RR\times\mathrm{Dom}(\cA)$,
$\psi\neq 0$, such that $\cA\psi=\lambda\psi$.  A second
representation theorem (see~\cite[Chapter 6, Theorem 2.23]{KatoBook})
expresses the bilinear form in terms of the self-adjoint and positive
definite square-root of $\cA$, $\cA^{1/2}$ (see~\cite[Chapter 5,
Theorem 3.35]{KatoBook}),
\begin{align*}
  B(v,w)=(\cA^{1/2}v,\cA^{1/2}w)\mbox{ for all }v,w\in\mathrm{Dom}(\cA^{1/2})=\cH~,
\end{align*}
and we see that $\enorm{v}=\|\cA^{1/2}v\|_0$. 

Let $\spec(\cA)$ denote the spectrum of $\cA$.  Given a finite subset
$\Lambda\subset\spec(\cA)$, let
\begin{align*}
E(\Lambda)=\mathrm{span}\{\psi\in\mathrm{Dom}(\cA):\,\cA\psi=\lambda\psi\mbox{
  for some }\lambda\in\Lambda\}
\end{align*}
be the associated invariant subspace.  Let $S(\Lambda)$ be the
$L^2(\Omega)$-orthogonal projector onto $E(\Lambda)$.  When
$\Lambda=\{\lambda\}$, we use $E(\lambda)$ and $S(\lambda)$.
It is well-known that $S(\Lambda)$ is also the orthogonal
projector onto $E(\lambda)$ with respect to the energy inner-product.
These orthogonal projection properties are stated as best
approximation results in the following proposition.
\begin{proposition}\label{ClosestPoint}
  Let $\Lambda\subset\spec(\cA)$ be a finite set, $E=E(\Lambda)$ and
  $S=S(\Lambda)$.  For any $v\in\cH$, it holds that
  $\norm{(I-S)v}=\inf_{w\in E}\norm{v-w}$, where $\norm{\cdot}$
  denotes either the $L^2$ or energy norm.  In the case of the $L^2$
  norm, we may allow $v\in L^2(\Omega)$.
\end{proposition}

Taking $\Lambda$ and $S=S(\Lambda)$ as above, let $\hat\mu\not\in
\spec(\cA)\setminus\Lambda$ be a non-zero real number, and $\hat\phi\in\cH$.
It can be seen in the proof of~\cite[Proposition 2]{GGMO} that
\begin{align}\label{KeyIdentity}
(I-S)\hat\phi=-[\cA^{1/2}(\hat\mu-\cA')^{-1}(I-S)][\cA^{1/2}(\hat\phi-\cA^{-1}(\hat\mu\hat\phi))]~
\end{align}
where $\cA'=\cA(I-S)$.  It follows that
\begin{align*}
\enorm{(I-S)\hat\phi}=\|\cA^{1/2}(I-S)\hat\phi\|_0=
\|\cA(\hat\mu-\cA')^{-1}(I-S)][\cA^{1/2}(\hat\phi-\cA^{-1}(\hat\mu\hat\phi))\|_0~.
\end{align*}
Since all of the operators in~\eqref{KeyIdentity} commute, we also
have
\begin{align*}
\|(I-S)\hat\phi\|_0=\|\cA(\hat\mu-\cA')^{-1}(I-S)][(\hat\phi-\cA^{-1}(\hat\mu\hat\phi))\|_0~.
\end{align*}
From these identities, we obtain the estimates

\begin{align}\label{EigenvectorErrors1}
\|(I-S)\hat\phi\|\leq&C(\hat\mu,\Lambda)
\|\hat\phi-\cA^{-1}(\hat\mu\hat\phi)\|~,
\end{align}
where $\norm{\cdot}$ denotes either the $L^2$ or energy norms, and
the constant $C(\hat\mu,\Lambda)$ is given by
\begin{align}\label{ErrorConstant1}
C(\hat\mu,\Lambda)=\|\cA(\hat\mu-\cA')^{-1}(I-S)\|_0=\|\cA'(\hat\mu-\cA')^{-1}\|_0=
\max_{\xi\in(\spec{\cA}\setminus\Lambda)\cup\{0\}}\frac{\xi}{|\xi-\hat\mu|}~.
\end{align}
The final identity can be found, for example, in~\cite[Chapter 5,
Section 3.5]{KatoBook}, and uses the fact that $\spec{\cA'}=(\spec{\cA}\setminus\Lambda)\cup\{0\}$. 
If $\Lambda=\{\lambda\}$, we use
$C(\hat\mu,\lambda)$ for this constant.

Now let $E=E(\Lambda)$, with $\dim E=r$.  Suppose we are given a real subspace
$\hat{E}\subset\cH$ of dimension $r$, as well as an $r$-tuple of positive
numbers $(\hat\mu_1,\ldots,\hat\mu_r)$ with
$\hat\mu_i\not\in\spec(\cA)\setminus\Lambda$, and $\hat\mu_1\leq\cdots\leq\hat\mu_r$.  
Taking $\{\hat{\phi}_1,\ldots,\hat{\phi}_r\}$ as a basis of $\hat{E}$,
we identify $\hat\mu_i$ with
$\hat\phi_i$.  It is natural to think of
$\hat\Lambda=\{\hat\mu_1,\ldots,\hat\mu_r\}$ and $\hat{E}$ as
approximations of $\Lambda$ and $E$ obtained by an $hp$-finite element
procedure, and we will do so later, but for now we work with the given
level of generality.  Of particular interest in our discussion is the
relative error in energy norm between $\hat{v}\in\hat{E}$ and its
projection $S\hat{v}\in E$.  Letting $G,H\in\mathbb{R}^{r\times r}$ be
the Gram matrices given by
\begin{align}\label{GramMatrices}
G_{ij}=B(\hat{\phi}_j,\hat{\phi}_i) \quad,\quad H_{ij}=B((I-S)\hat{\phi}_j,(I-S)\hat{\phi}_i) ~,
\end{align}
and $\mb{v}\in\RR^r$ be the coefficient vector of $\hat{v}$ with
respect to the (ordered) basis $(\hat{\phi}_1,\ldots,\hat{\phi}_r)$, we have
\begin{align}\label{RelativeErrorIdentity}
\frac{\enorm{(I-S)\hat v}^2}{\enorm{\hat v}^2}=\frac{\mb{v}^t H\mb{v}}{\mb{v}^t G\mb{v}}~.
\end{align}
This naturally leads to our first key result.

\begin{theorem}\label{TraceTheorem}  We have the eigenvector
        error trace estimate 
	\begin{align}\label{TraceEstimate}
	\sup_{\hat v\in\hat{E}}\frac{\enorm{(I-S)\hat v}^2}{\enorm{\hat v}^2}\leq\frac{[C(\hat\Lambda,\Lambda)]^2}{\lambda_{\min}(G)}
	\,\sum_{j=1}^r\enorm{\hat{\phi}_j-\cA^{-1}(\hat\mu_j\hat{\phi}_j)}^2~,
	\end{align}
	where
        $C(\hat\Lambda,\Lambda)=\max\{C(\hat\mu_j,\Lambda):\,1\leq
        j\leq r\}$.  If we further assume that
        $B(\hat\phi_i,\hat\phi_j)=\hat\mu_i\delta_{ij}$, then we have the following
        modification of~\eqref{TraceEstimate},
	\begin{align}\label{TraceEstimate2}
         \sup_{\hat v\in\hat{E}} \frac{\enorm{(I-S)\hat v}^2}{\enorm{\hat v}^2}\leq [C(\hat\Lambda,\Lambda)]^2
          \,\sum_{j=1}^r\frac{\enorm{\hat{\phi}_j-\cA^{-1}(\hat\mu_j\hat{\phi}_j)}^2}{\hat\mu_j}~,
	\end{align}  
as well as the eigenvalue error trace estimate,
\begin{align}\label{EigenvalueTrace}
\sum_{j=1}^r(\hat\mu_j-\mu_j)\leq [C(\hat\Lambda,\Lambda)]^2
          \,\sum_{j=1}^r\enorm{\hat{\phi}_j-\cA^{-1}(\hat\mu_j\hat{\phi}_j)}^2~,
\end{align}
where $\Lambda=\{\mu_1,\ldots,\mu_r\}$, with $\mu_1\leq\cdots\leq\mu_r$.
\end{theorem}
\begin{proof}
  The ratio in~\eqref{RelativeErrorIdentity} is clearly controlled by
  the eigenvalues of $G^{-1}H$, and a simple upper-bound is given by
  $\mathrm{trace}(H)/\lambda_{\min}(G)$.  Combining this
  with~\eqref{EigenvectorErrors1} and~\eqref{ErrorConstant1} yields
  the bound~\eqref{TraceEstimate}.  Under the further assumptions on
  $\hat\phi_j$ and $\hat\mu_j$, $G$ is diagonal, and we instead
  bound~\eqref{RelativeErrorIdentity} by $\mathrm{trace}(G^{-1}H)$ to
  obtain~\eqref{TraceEstimate2}.  For the eigenvalue estimate, 
   let $\{\phi_1,\ldots,\phi_r\}$ be an orthonormal eigenbasis of $E$, with
  $\mu_j=\enorm{\phi_j}^2$. We have
\begin{align*}
\sum_{i=1}^r\enorm{(I-S)\hat\phi_i}^2&=\sum_{i=1}^r\left(
\enorm{\hat\phi_i}^2-\sum_{j=1}^r\mu_j[(\phi_j,\hat\phi_i)]^2\right)\\
&=
\sum_{i=1}^r\hat\mu_i-\sum_{j=1}^r\mu_j \sum_{i=1}^r[(\phi_j,\hat\phi_i)]^2
\geq \sum_{i=1}^r\hat\mu_i-\sum_{j=1}^r\mu_j ~.
\end{align*}
The bounds $\enorm{(I-S)\hat\phi_i}\leq C(\hat\mu_j,\Lambda)
\|\hat\phi_j-\cA^{-1}(\hat\mu_j\hat\phi_j)\|$ complete the proof.
\end{proof}

\begin{remark}\label{GapRemark}
The subspace gap (cf.~\cite[Chapter 4, Section 2]{KatoBook}) is a
standard measure of  distance between subspaces.
The ``Pair of Projectors Alternative'' \cite[Chapter 1, Theorem
6.34]{KatoBook}, implies that, if
        $\sup_{\hat v\in\hat{E}}\enorm{(I-S)\hat v}/\enorm{\hat
            v}<1$, then 
\begin{align}\label{SubspaceGap}
\mathrm{gap}(E,\hat{E})=
\sup_{\hat v\in\hat{E}}\inf_{v\in E}\frac{\enorm{v-\hat  v}}{\enorm{\hat v}}
=\sup_{v\in E}\inf_{\hat v\in\hat{E}}\frac{\enorm{v-\hat v}}{\enorm{v}}<1~.
\end{align}
More generally, the gap between two subspaces $M,N$ of $\cH$, with respect
to the energy norm is
\begin{align*}
\mathrm{gap}(M,N)=\max\left\{
\sup_{w \in M}\inf_{v\in N}\frac{\enorm{v-w}}{\enorm{w}}\,,\,\sup_{v\in N}\inf_{w\in M}\frac{\enorm{v-w}}{\enorm{v}}\right\}~.
\end{align*}
If $P_M$ and $P_N$ are the corresponding orthogonal projectors (with
respect to the energy inner-product), then
$\mathrm{gap}(M,N)=\enorm{P_M-P_N}$, so we see that the gap provides a
metric between subspaces.  In fact, when $M$ and $N$ have the same finite
dimension, $\mathrm{gap}(M,N)$ is the sine of the largest principle
angle between the these subspaces (cf.~\cite{Knyazev2010}).  
A natural alternative to the subspace gap is to measure the distance
between the corresponding orthogonal projectors using a
Hilbert-Schmidt norm.  This is the approach taken
in~\cite{Cances2020}, for example.
\end{remark}

\begin{remark}\label{ConstantRemark}  Suppose that
  $\Lambda,\hat\Lambda\subset (a,b)$ for some $0<a<b$, and
  $\spec(A)\setminus\Lambda\subset(0,a]\cup[b,\infty)$.  Then we have
\begin{align}\label{ConstantEstimate}
C(\hat\Lambda,\Lambda)=\max\{C(\hat\mu_1,\Lambda),
  C(\hat\mu_r,\Lambda)\}\leq \max\left\{\frac{a}{\hat\mu_1-a}, \frac{b}{b-\hat\mu_r}\right\}~.
\end{align}
\end{remark}

For $f\in L^2(\Omega)$, we define $u(f)\in \cH$ and $\hat{u}(f)\in V$ by
\begin{align}\label{SourceProblems}
B(u(f),v)=(f,v)_0\mbox{ for all }v\in\cH\quad,\quad
             B(\hat{u}(f),v)=(f,v)_0\mbox{ for all }v\in V~.
\end{align}
Now suppose that $(\hat\mu_j,\hat\phi_j)$ is a solution
of~\eqref{DiscVarEig} for $1\leq j\leq r$.
Taking $f_j=\hat\mu_j\hat\phi_j$, we have
$u(f_j)=\cA^{-1}(\hat\mu_j\hat\phi_j)$ and $\hat{u}(f_j)=\hat\phi_j$, and we
rephrase~\eqref{TraceEstimate2} and~\eqref{EigenvalueTrace} as
\begin{align}\label{TraceEstimates3a}
\sup_{\hat v\in\hat{E}} \frac{\enorm{(I-S)\hat v}^2}{\enorm{\hat v}^2}&\leq [C(\hat\Lambda,\Lambda)]^2
          \,\sum_{j=1}^r\frac{\enorm{u(f_j)-\hat{u}(f_j)}^2}{\hat\mu_j}~,\\
\sum_{j=1}^r(\hat\mu_j-\mu_j)&\leq [C(\hat\Lambda,\Lambda)]^2
          \,\sum_{j=1}^r \enorm{u(f_j)-\hat{u}(f_j)}^2~. \label{TraceEstimates3b}
\end{align}
These forms of the estimates emphasize that eigenvalue and eigenspace
errors are controlled by discretization errors of source problems
whose data are drawn from the discrete eigenpairs, and
we will return to them in our development of practical \textit{a
  posteriori} estimates for eigenvalue and eigenspace
errors in Section~\ref{HPDiscretization}.  We also note that, in this
setting, $\hat\mu_j\geq \mu_j$.

The upper-bound on the subspace gap provided
in~\eqref{TraceEstimates3a}, as well as its computable counterpart
in~\eqref{TraceEstimates4a} are theoretically convenient
over-estimates, which may be pessimistic when $r$ is large.
In fact, they might seem more natural as bounds on a Hilbert-Schmidt
type measure of subspace error (see Remark~\ref{GapRemark}).  The
recent contribution~\cite{Cances2020} takes this approach.
If we had a decent computable approximation $\tilde{H}$ of the Gram
matrix $H$ from~\eqref{RelativeErrorIdentity}, we could compute the
eigenvalues of generalized eigenvalue problem
$\tilde{H}\mb{x}=\tilde\kappa G\mb{x}$ directly, and not be restricted
to trace-type estimates such
as~\eqref{TraceEstimate},~\eqref{TraceEstimate2}
or~\eqref{TraceEstimates3a}.  The largest of these eigenvalues would
then provide an estimate of the subspace gap~\eqref{SubspaceGap}.
In Section~\ref{sec:Bridge Configurations}, we illustrate how the
error estimates described in Section~\ref{HPDiscretization} enable
the computation of such an approximation $\tilde{H}$ of $H$.

The following
Bauer-Fike estimate (cf.~\cite[Theorem 7.2.2]{GolubVanLoan}) provides
a measure of distance between the eigenvalues of $(H,G)$ and
$(\tilde{H},G)$ in terms of matrix norms.
\begin{proposition}\label{BauerFike}
Let $\tilde{H}$ be a positive semi-definite approximation of the Gram
matrix $H$ from~\eqref{RelativeErrorIdentity}.  If $\tilde\kappa$ is an
eigenvalue for the pair $(\tilde{H},G)$, then 
\begin{align*}
\min_{\kappa\in\spec(H,G)}|\kappa-\tilde\kappa|\leq \|G^{-1/2}\|_p^2\|H-\tilde{H}\|_p~,
\end{align*}
for any matrix $p$-norm, where $G^{-1/2}$ is the (unique) positive definite
square root of $G^{-1}$.  
\end{proposition}
We note that, if
$G=\mathrm{diag}(\hat\mu_1,\cdots,\hat\mu_r)$, then
$\|G^{-1/2}\|_p^2=(\min\{\hat\mu_j\})^{-1}$.
It is clear that the roles of $H$ and $\tilde{H}$ can be reversed
in Proposition~\ref{BauerFike}, so we actually have a bound on the
Hausdorff distance between $K=\spec(H,G)$ and
$\tilde{K}=\spec (\tilde{H},G)$, 
\begin{align}\label{ApproxDef_Hausdorff}
\mathrm{dist}(K,\tilde{K})\doteq
\max\left\{\max_{\tilde\kappa\in\tilde{K}}\min_{\kappa\in K}|\kappa-\tilde{\kappa}|\,,\,
 \max_{\kappa\in K}\min_{\tilde\kappa\in\tilde{K}}|\kappa-\tilde{\kappa}|
  \right\}\leq
\|G^{-1/2}\|_p^2\|H-\tilde{H}\|_p~.
\end{align}
In Section~\ref{HPDiscretization}, after we have properly introduced the
approximate error functions $\varepsilon_j$, we
illustrate~\eqref{ApproxDef_Hausdorff} for the Slit Disk problem and
the choice $\tilde{H}_{ij}=B(\varepsilon_j, \varepsilon_i)$, in Example~\ref{SlitDiskRevisited}.

\begin{remark}\label{ApproximationDefects}
  In~\cite{Grubisic2006b,Grubisic2009,Bank2013}, the authors develop
  an a posteriori eigenvalue and eigenvector error analysis based on
  \textit{approximation defects}, which are essentially the square
  roots of the eigenvalues of the pair $(H,G)$.
\end{remark}

\begin{remark}
	A natural question would be the choice of optimal $p$. Matrix $p$-norms in general do not have a monotonic relationship and the choice of $2$ norms presents an easily computable norm. More on evaluating other matrix $p$ norms can be found in \cite{Higham1992}.
\end{remark}

\section{p and hp Finite Element Discretization, A Posteriori Estimates}
\label{HPDiscretization}

Let $\Omega\subset\RR^2$ be a open, bounded domain, with Lipschitz
boundary $\partial\Omega$, and let $\cT=\{T\}$ be a conforming
partition of $\Omega$ into convex (curvilinear) triangles and
quadrilaterals, which we call a \textit{mesh} or
\textit{triangulation}, see Figure~\ref{fig:SlitDisk}.  We do not
impose any restriction on the number of curved edges.  Any curved
elements are handled using standard blending function techniques
(cf.~\cite{szabo}).  In order to reduce the level of technicality in
describing the families of finite element spaces that we will
consider, we state them for only in the case of polygonal domains,
partitioned into triangles and quadrilaterals.  

For a given element $T$ and non-negative integer $m$, we define the
local polynomial space $\QQ_m(T)$ as follows.  If $T$ is a triangle,
then $\QQ_m(T)$ consists of the polynomials of total degree $\leq m$,
so $\dim \QQ_m(T)=(m+2)(m+1)/2$.  If $T$ is a quadrilateral, then
$\QQ_m(T)$ consists of polynomials of degree $\leq m$ in each
variable, so $\dim \QQ_m(T)=(m+1)^2$.
For a given triangulation, $\cT$, let $\mathbf{p}:\cT\to\NN$ be a function that
assigns a positive integer to each element $T\in\cT$. This map
is called a $p$-vector.
We define the corresponding finite element space
\begin{align}
V=V(\cT,\mathbf{p})=\{v\in \cH:\,v_{\vert_{T}}\in\QQ_{\mathbf{p}(T)}(T)\mbox{ for all }T\in\cT\}~.
\end{align}
We note that $V\subset C(\overline\Omega)$.  

Let $\cF=\{\cT_\ell\}$ be a family of nested meshes obtained from
successive refinements of an initial coarse mesh, where the index
$\ell\geq 0$ refers to a refinement level.  Of particular interest to
us in this work are eigenproblems, such as that in
Example~\ref{SlitDiskExample}, for which it is known that certain
eigenfunctions will be singular (e.g. have unbounded derivatives) at
particular points in the domain.  Such points are commonly referred to
as \textit{singular points}, and in the case of the Laplace operator,
occur at points on the boundary where there are non-convex corners,
and where there is a shift in the type of boundary condition
(e.g. from a Dirichlet condition to a Neumann condition).  We
explore examples having this second type of singular points
extensively in the Section~\ref{sec:Exper}.  The asymptotic behavior
of such singularities in the vicinity of singular points is
well-understood (cf.~\cite{Grisvard1985,Grisvard1992,Kozlov1997}), and
based on such a priori knowledge various refinement approaches have
been proposed that involve a geometric grading of element sizes toward
singular points that takes into account this a priori knowledge of the
singularity strength~\cite[Section 4.5]{Schwab1998}.  Beginning with a coarse mesh
$\cT_0$ in which the vertex graph distance between singular points
(i.e., the minimal number of edges in a path connecting these points)
is at least two, the mesh grading approach is implemented using
element-level replacement rules employing exact geometry description
as described in~\cite{Hakula2013}.

Given such a family of meshes, we distinguish two families of finite
element spaces defined on them.  We refer to the first as the
$p$-method family because it uses a fixed polynomial degree for every
element in the mesh.  For this family, the polynomial degree $p$ is
chosen and applied to each element in the $p$th mesh in the family,
$\cT_p\in\cF$, i.e. $\mb{p}(T)=p$ for all $T\in\cT_p$.  We denote the
finite element spaces in this family by $V_{1,p}$, and use $4\leq p\leq 12$ for our experiments. 
We note that the spaces are nested, $V_{1,p}\subset V_{1,p+1}$.
We refer to the second family as the $hp$-family because it uses
variable polynomial degrees in the mesh.  For the second family, given a
polynomial degree $p$, the mesh $\cT_p$ is chosen as in the first
family, but polynomial degrees are no longer assigned uniformly throughout
the mesh.  All elements touching a singular point are assigned
polynomial degree $1$, the next layer of elements are assigned
polynomial degree $2$, and so on, until polynomials of degree $p$ are
achieved at the $p$th layer.  Any elements that are greater than $p$
layers away from all singular points are also assigned polynomial
degree $p$.   The initial mesh and refinement scheme ensures that
there is no ambiguity in how polynomial degrees are assigned to each
element. The element layers are created by nested application of 
the same replacement rule
on every element touching a singular point. At each step, only the 
elements touching the singular point created at the previous one are
refined making the bookkeeping of the layers simple. 
This is illustrated in Figure~\ref{fig:SlitDisk}.
We denote the finite element spaces in this  by $V_{2,p}$,
again using $4\leq p\leq 12$ for our experiments. 
As before, the spaces are nested, $V_{2,p}\subset V_{2,p+1}$, and we
also note that $V_{2,p}\subset V_{1,p}$.  

In practice, three types of polynomial functions are distinguished on
an element: \textit{vertex functions}, which vanish on all vertices
except one; \textit{edge functions}, which vanish on all edges except
one; and \textit{element functions} (interior bubble functions), which
vanish on all edges.  On the global (mesh) level, vertex functions are
supported in the patch of elements sharing that vertex, edge functions
are supported in the (one or two) elements sharing an edge, and
element functions are supported in a single element.  There are
well-established techniques for constructing \textit{hierarchical
  bases} for $\QQ_p(T)$ (cf.~\cite{Schwab1998}), starting from a basis
of $\QQ_1(T)$ (vertex functions), augmenting it with edge functions
from $\QQ_2(T)$ to form a basis for $\QQ_2(T)$, further augmenting
this with edge and element functions from $\QQ_3(T)$ to form a basis
for $\QQ_3(T)$, and so on.  This distinction between the types of
polynomial functions enables one to build elements in which the
degrees of the element functions may differ from those of the edge
functions, and the degree used on one edge may differ from that used
on another.  In fact, this is precisely what is done in the
$hp$-family, $V_{2,p}$, to allow for variable $p(T)$.  In particular,
when $T$ and $T'$ are adjacent elements whose assigned polynomial
degrees differ by one, say $p(T)=m$ and $p(T')=m+1$, the polynomial
degree of the edge functions associated with their shared edge is
taken to be $m+1$.  The use of hierarchical bases, and the distinction
between edge and element functions plays a prominent role in the type
of a posteriori error estimates that we now discuss.

\subsection{A Posteriori Estimates}\label{sec:AuxiliarySubspace} 
As suggested in Section~\ref{Theory}, most clearly in the eigenvalue
and eigenvector error
estimates~\eqref{TraceEstimates3a}-\eqref{TraceEstimates3b}, we see
how such error estimates are built upon those for source problems.
More specifically we see that an a posteriori estimate of
$\enorm{u(f_j)-\hat{u}(f_j)}$, where the ``source'' $f_j=\hat\mu_j\hat\phi_j$ is
obtained from the approximate eigenpair $(\hat\mu_j,\hat\phi_j)$,
provides a measure of how far $\hat\phi_j$ is from a true eigenvector
and how far $\hat\mu_j$ is from a true eigenvalue.  The field of a
posteriori error estimation for source problems, particularly for
the reaction-diffusion operators we consider here, is quite mature, so
we have many well-documented methods for estimating 
$\enorm{u(f_j)-\hat{u}(f_j)}$.  The one we have chosen for the present
work is based on the principle of hierarchical bases, that seems
particularly well-suited to the $p$- and $hp$-settings.  What we use
is described in detail, and rigorously tested, in~\cite{HNO}, so we
provide a general overview here, and focus on how it is used in the
eigenvalue/vector context.

Given $f\in L^2(\Omega)$, the exact and finite element solutions,
$u(f)\in\cH$ and $\hat{u}(f)\in V$, satisfy
\begin{align*}
B(u(f),v)=(f,v)\mbox{ for all }v\in\cH\quad,\quad
  B(\hat{u}(f),v)=(f,v) \mbox{ for all }v\in V~.
\end{align*}
We compute an approximate error function $\varepsilon(f)\in W$ in an
\textit{auxiliary subspace} $W\subset\cH$ as the projection of
$u(f)-\hat{u}(f)$ onto $W$,
\begin{align}\label{ErrorEquation}
B(\varepsilon(f),v)= B(u(f)-\hat{u}(f),v)=(f,v)-B(\hat{u}(f),v)\mbox{
  for all }v\in W~.
\end{align}
The \textit{error space} $W$ is chosen so that $V\cap W=\{0\}$, $W$ is
defined on the same mesh as $V$, and $V\oplus W$ is a richer
approximation space on this mesh.  We adopt the approach suggested
in~\cite{HNO} for problems in 2D, which we describe at the level of
elements.  If element functions of degree $m$ are used on an element
$T$ in $V$, then element functions of degree $m+2$ are used on this
same element in $W$.  If edge functions of degree $m$ are used on an
edge $e$ in $V$, then edge functions of degree $m+1$ are used on this
same edge in $W$.  We slightly rephrase~\cite[Theorem 1.4]{HNO} in our
context for the energy norm.  We take $\cE$ to be the set of edges of
the mesh that are not on the Dirichlet part of the boundary, and
define the \textit{volumetric residual},
$R_T=f-(-\nabla\cdot A\nabla \hat{u}(f)+b\hat{u}(f)) _{\vert_T}$.
When $e\in\cE$ is an interior edge, we define the \textit{edge residual} as
$r_e=(A\nabla\hat{u}(f)\cdot\mb{n}_T)_{\vert_T}+(A\nabla\hat{u}(f)\cdot\mb{n}_{T'})_{\vert_T'}$,
where $T$ and $T'$ are the cells sharing this egde, and $\mb{n}_T$ and
$\mb{n}_{T'}$ are their outward unit normals.  For a Neumann boundary
edge, we define the edge residual as
$r_e=(A\nabla\hat{u}(f)\cdot\mb{n}_T)_{\vert_T}$.  With these
definitions in hand, we can state the theorem.
\begin{theorem}\label{Thm1_4HNO}
There is a constant $c$, depending on the shape-regularity of $\cT$
and the polynomial degree $p$ such that
\begin{align*}
\enorm{\varepsilon(f)}\leq \enorm{u(f)-\hat{u}(f)}\leq c\left(\enorm{\varepsilon(f)}+\mathrm{osc}(R,r,\cT)\right)~,
\end{align*}
where the residual oscillation is defined by
\begin{align*}
[\mathrm{osc}(R,r,\cT)]^2=\sum_{T\in\cT}h_T^2\inf_{\kappa\in\QQ_{p-1}(T)}\|R_T-\kappa\|_{L^2(T)}^2+\sum_{e\in\cE}|e|\inf_{\kappa\in\QQ_{p-1}(e)}\|r_e-\kappa\|_{L^2(e)}^2~,
\end{align*}
where $h_T$ and $|e|$ are the diameter of $T$ and length of the edge $e$,
respectively.
\end{theorem}
The proof given in~\cite{HNO} was given for simplicial meshes, but its
performance was rigorously tested for more general meshes containing
both (curvilinear) triangles and quadrilaterals in 2D, and
  hexahedral meshes in 3D.  Although~\cite{HNO}
provides compelling numerical evidence that $c$ is independent of $p$,
such independence has not been theoretically established.

\begin{remark}\label{SpecialOscillation}
If $A$ is piecewise constant on $\Omega$, and constant on each
$T\in\cT$, then the residual oscillation term in Theorem~\ref{Thm1_4HNO}
simplifies to
\begin{align*}
[\mathrm{osc}(R,r,\cT)]^2=\sum_{T\in\cT}h_T^2\inf_{\kappa\in\QQ_{p-1}(T)}\|b
  \hat{u}(f)-\kappa\|_{L^2(T)}^2~.
\end{align*}
In our examples, $b=0$ as well.  In this case, there is no residual oscillation
at all, and the error estimate of Theorem~\ref{Thm1_4HNO} becomes
\begin{align}\label{ErrorEquivalence}
\enorm{\varepsilon(f)}\leq \enorm{u(f)-\hat{u}(f)}\leq c\enorm{\varepsilon(f)}~.
\end{align}
\end{remark}

In the eigenvalue context, suppose we have computed approximate
eigenpairs $\{(\hat\mu_i,\hat\phi_i):\,1\leq i\leq r\}$ in $V$, with
$B(\hat\phi_i,v)=\hat\mu_i(\hat\phi_i,v)$ for all $v\in V$ and 
$(\hat\phi_i,\hat\phi_j)=\delta_{ij}$.  Our approximation $\tilde{H}$ of the
matrix $H$ in~\eqref{GramMatrices} is given by 
\begin{align}\label{HApprox}
\tilde{H}_{ij}=B(\varepsilon_j,\varepsilon_i)\quad,\quad
  \varepsilon_k=\varepsilon(f_k)
\quad,\quad f_k=\hat\mu_k\hat\phi_k~.
\end{align}
The matrix $G$ is diagonal, $G=\mathrm{diag}(\hat\mu_1,\ldots,\hat\mu_r)$.
Assuming piecewise constant $A$ and $b=0$, as in
Remark~\ref{SpecialOscillation}, in order to state eigenvector and
eigenvalue error estimates without residual oscillation terms, we have
\begin{align}\label{TraceEstimates4a}
\sup_{\hat v\in\hat{E}} \frac{\enorm{(I-S)\hat v}^2}{\enorm{\hat v}^2}&\leq c^2[C(\hat\Lambda,\Lambda)]^2
          \,\sum_{j=1}^r\frac{\enorm{\varepsilon_j}^2}{\hat\mu_j}~,\\
\sum_{j=1}^r(\hat\mu_j-\mu_j)&\leq c^2[C(\hat\Lambda,\Lambda)]^2
          \,\sum_{j=1}^r \enorm{\varepsilon_j}^2~. \label{TraceEstimates4b}
\end{align}
For convenience, the corresponding estimates for the approximation errors
associated with a single ($r=1$), simple, eigenvalue $\mu_1=\lambda_k$
with eigenvector $\phi_1=\psi_k$
and the computed eigenpair
$(\hat\mu_1,\hat\phi_1)=(\hat\lambda_k,\hat\psi_k)$ are
\begin{align}\label{TraceEstimates5}
\inf_{v\in \mathrm{span}\{\psi_k\}}\enorm{\hat\psi_k-v}\leq C_k
          \,\enorm{\varepsilon(\hat\lambda_k\hat\psi_k)}\quad,\quad
\hat\lambda_k-\lambda_k\leq C_k^2\,\enorm{\varepsilon(\hat\lambda_k\hat\psi_k)}^2~,
\end{align}
where $C_k=c \,C(\hat\lambda_k,\lambda_k)$.

\begin{example}\label{SlitDiskRevisited}
  We revisit the Slit Disk example, Example~\ref{SlitDiskExample} from
  Section~\ref{Intro}, computing the approximate eigenpairs
  $(\hat\lambda_k,\hat\psi_k)$, $1\leq k\leq 60$, on a sequence of
  $p$-version and $hp$-version finite element spaces associated with a
  family of meshes that are strongly graded toward the origin (see
  Figure~\ref{fig:SlitDisk}).  These meshes/spaces, which provide good
  approximation of features (e.g. singularities) of eigenmodes near
  the origin, but for smaller $p$ do not capture oscillatory behavior
  farther away from the origin nearly as well, were deliberately
  designed to demonstrate phenomena such as the mixing of modes
  observed in Figure~\ref{fig:SlitDiskConvergence52_53}.  
\begin{figure}
	\centering
	\subfloat[Mesh at 100\%.]
	{\label{fig:SlitDiskMesh1}\includegraphics[width=0.45\textwidth]{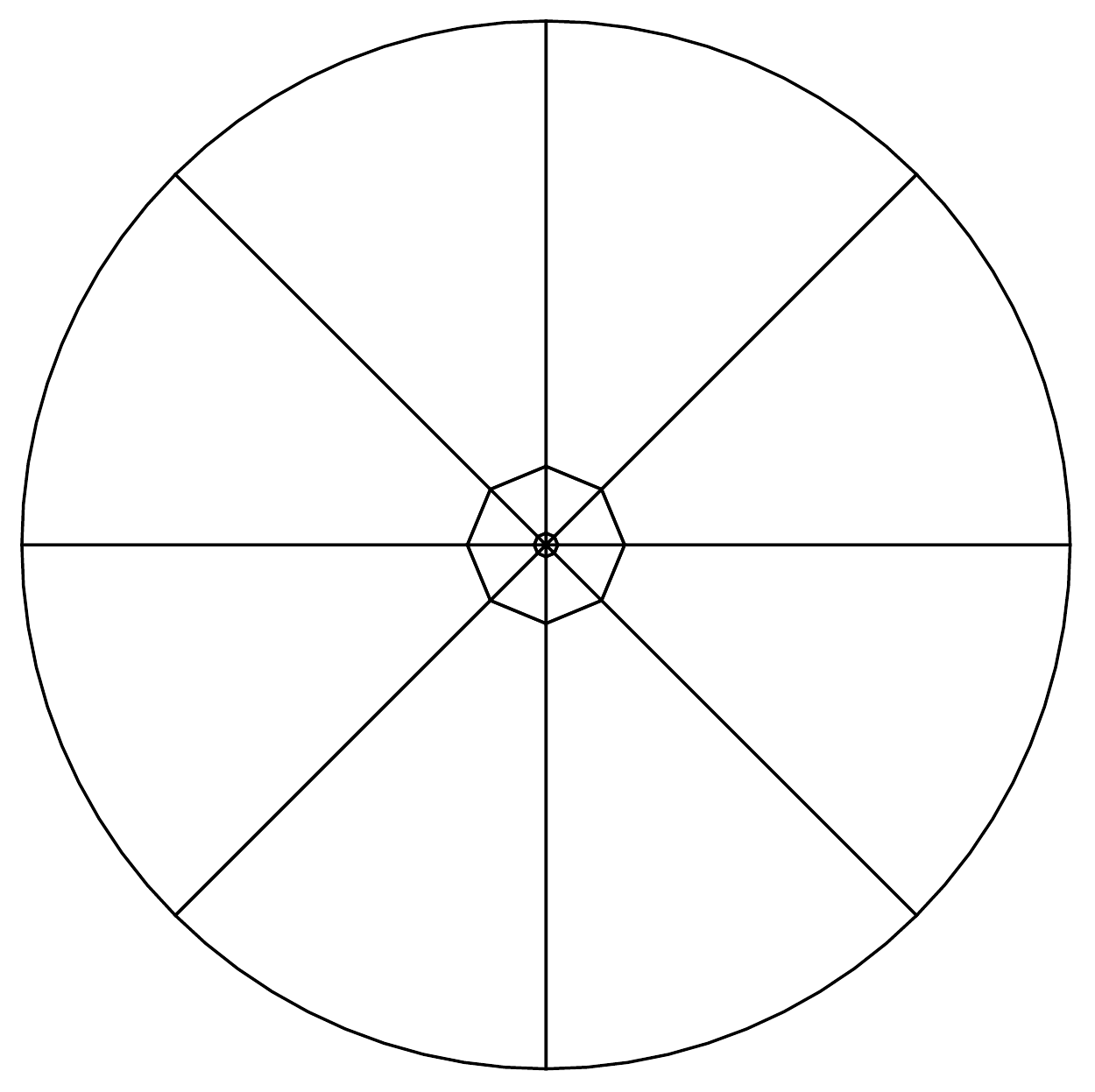}}
\quad
	\subfloat[Central portion of mesh at 1200\%.]
	{\label{fig:SlitDiskMesh12}\includegraphics[width=0.45\textwidth]{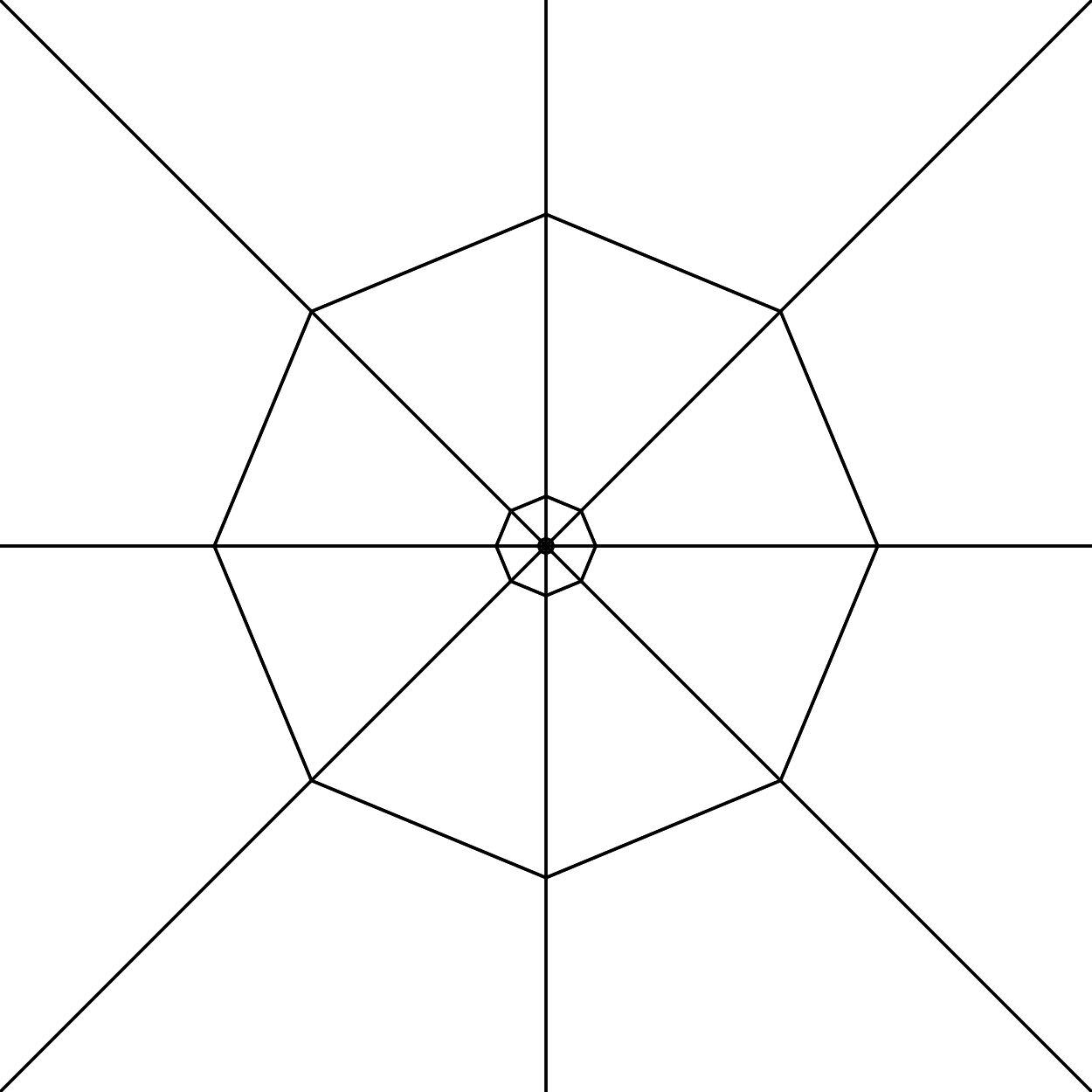}}
	\caption{A strongly graded mesh on the Slit Disk, and a close-up.}\label{fig:SlitDisk}
\end{figure}

We first consider the potential effects of this phenomena on the
quality of our computable estimates of eigenvalue and eigenvector
error. Our measures of ``quality'' are the \textit{effectivity
  ratios} (effectivities),
\begin{align}
\frac{\enorm{\varepsilon(\hat\lambda_k\hat\psi_k)}^2}{\hat\lambda_k-\lambda_k}
\quad,\quad
\frac{\|\varepsilon(\hat\lambda_k\hat\psi_k)\|_{0}}{\|\psi_k-\hat\psi_k\|_{0}}
\quad,\quad
\frac{\enorm{\varepsilon(\hat\lambda_k\hat\psi_k)}}{\enorm{\psi_k-\hat\psi_k}}
~.
\end{align}
Although our error estimates were only established for the energy
norm, we consider function error in $L^2$ as well.  The choice of
$\psi_k$ is normalized by taking $\|\psi_k\|_{0}=\|\hat\psi_k\|_{0}=1$
and $\alpha_k\doteq (\psi_k,\hat\psi_k)\geq 0$.
We note that $S\hat\psi_k=\alpha_k\psi_k$ in this case, so
$\|\psi_k-\hat\psi_k\|\neq \|\psi_k-S\hat\psi_k\|$ in either of the
two norms.  However, in either norm, we have 
$\|\psi_k-S\hat\psi_k\|\leq \|\psi_k-\hat\psi_k\|\leq
\sqrt{2}\|\psi_k-S\hat\psi_k\|$, with equality in the upper bound
achieved for the $L^2$-norm when $\alpha_k=0$, and equality approached
in the lower bound (for either norm) as $\alpha_k$ approaches $1$.
For these sequences of discretizations, $\alpha_k$ approached $1$ very
quickly, so there were no appreciable differences between effectivities
using $\|\psi_k-\hat\psi_k\|$ in the denominator versus using $\|\psi_k-S\hat\psi_k\|$.
In Figure~\ref{fig:SlitDiskClusterPvsHP}, we provide plots of the
effectivities for both families of discretizations, and $k=1,52,53$,
recalling that the computed eigenmodes $\hat\psi_{52}$ and
$\hat\psi_{53}$ do not begin to meaningfully approximate $\psi_{52}$
and $\psi_{53}$, respectively, until $p=7$.  The case $k=1$, for which
nothing unexpected happens, is considered merely as a comparative
baseline.  The poor effectivities of the estimates of the eigenmode
approximation errors for $k=52,53$ when $p<7$ stand out, and are not
surprising, because $\hat\psi_k$ is actually approximating $\psi_j$
for some $j\neq k$ when $p<7$.  In light of this, we also provide plots of the
ratios $\|\varepsilon_k(\hat\lambda_k\hat\psi_k)\|/\|\psi_j-\hat\psi_k\|$ for both norms, where
$\psi_j$ is also normalized as described above, as well as plots of
$\enorm{\varepsilon(\hat\lambda_k\hat\psi_k)}^2/(\hat\lambda_k-\lambda_j)$
for the eigenvalue $\lambda_j$ corresponding to $\psi_j$; these plots are given
in gray in Figure~\ref{fig:SlitDiskClusterPvsHP}.  
Complementary
eigenvalue and eigenmode convergence graphs are given in
Figure~\ref{fig:SlitDiskClusterPConvergence}.  Since the convergence
histories for both the $p$ and $hp$-families were very similar, only
those for the $p$-family are shown.   For $p\geq 6$, we observe a
``staircase'' pattern to the errors, where at first it would appear
that the error decreases only at odd $p$.  
This kind of staircase convergence phenomenon has been observed
elsewhere for $p$-method
  approximations on large elements (cf.~\cite[Figures 2.4 and
  2.5]{Babuska1992}).  In our case, we expect that this effect is
  present for $k=52,53$ because the corresponding eigenmodes oscillate
  within the large elements away from the origin---the behavior of the
  eigenmodes near the origin is resolved well by the highly graded
  mesh.
Since $\hat\psi_k$ does not approximate $\psi_k$ for $k=52,53$ when
$p\leq 6$, the corresponding errors in
Figure~\ref{fig:SlitDiskClusterPConvergence} are really only
meaningful for $p\geq 7$, at which point the errors take their first
significant drop and begin the odd-even staircase pattern.
\begin{figure}
	\centering
	\subfloat[$p$-version: Eigenvalue  effectivity.]
	{
	\includegraphics[width=0.45\textwidth]{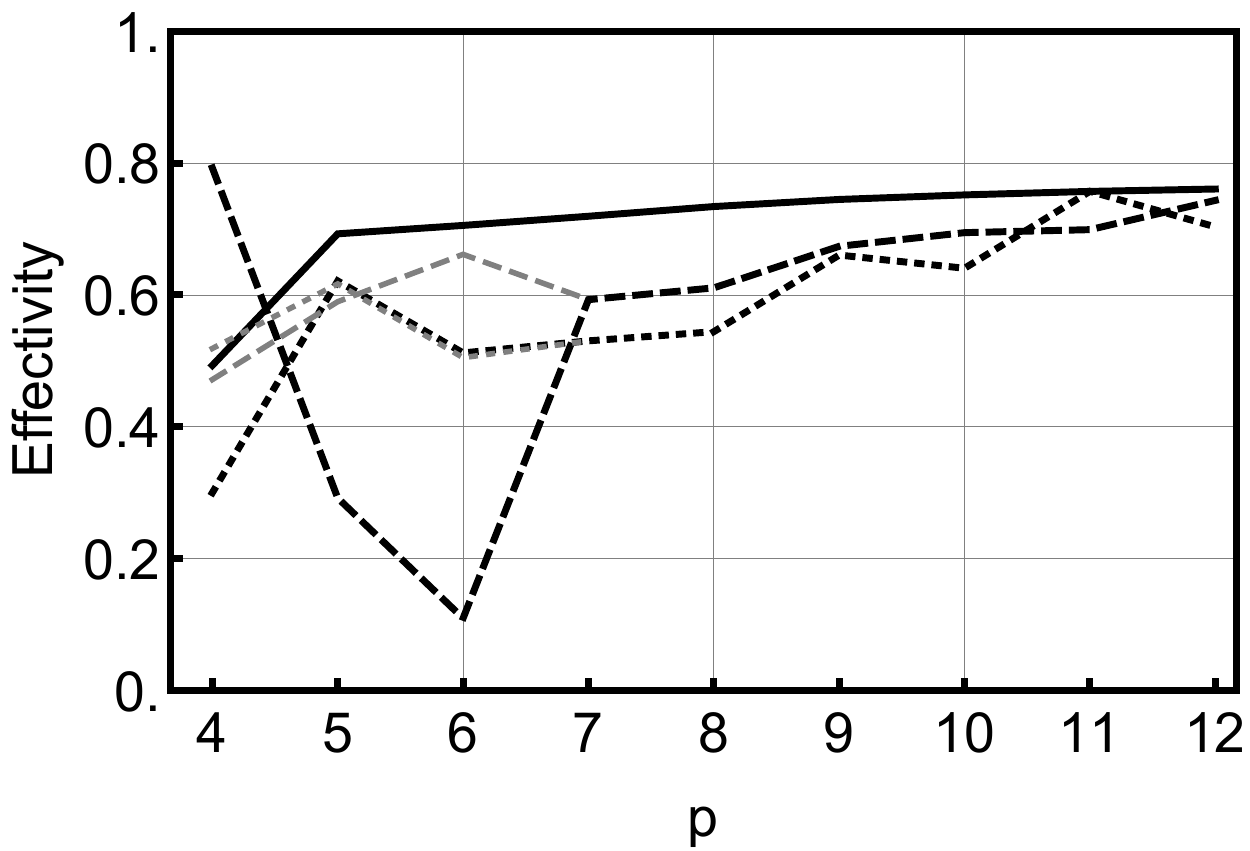}
	}\quad
	\subfloat[$hp$-version: Eigenvalue  effectivity.]
	{
	\includegraphics[width=0.45\textwidth]{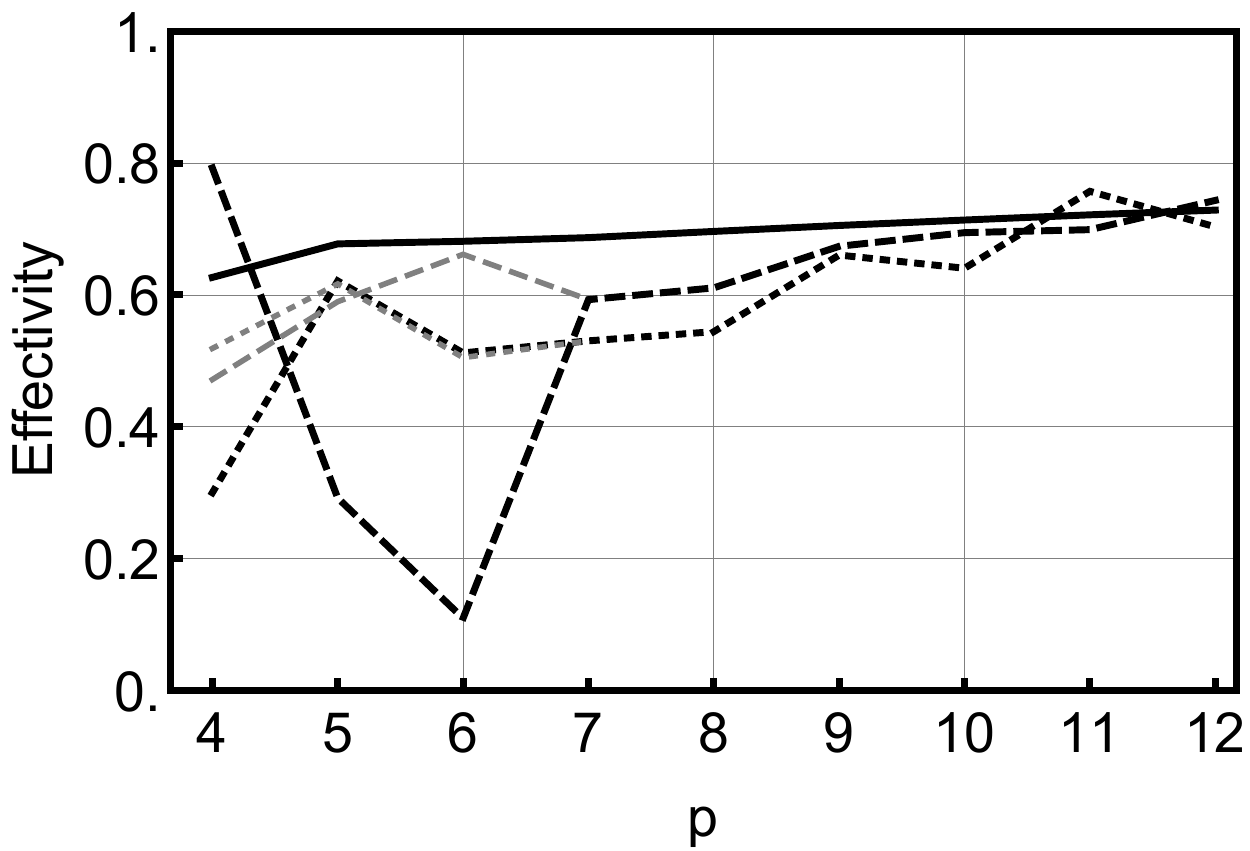}
	}\\
	\subfloat[$p$-version: Eigenvector  effectivity in $L^2$-norm.]
	{
	\includegraphics[width=0.45\textwidth]{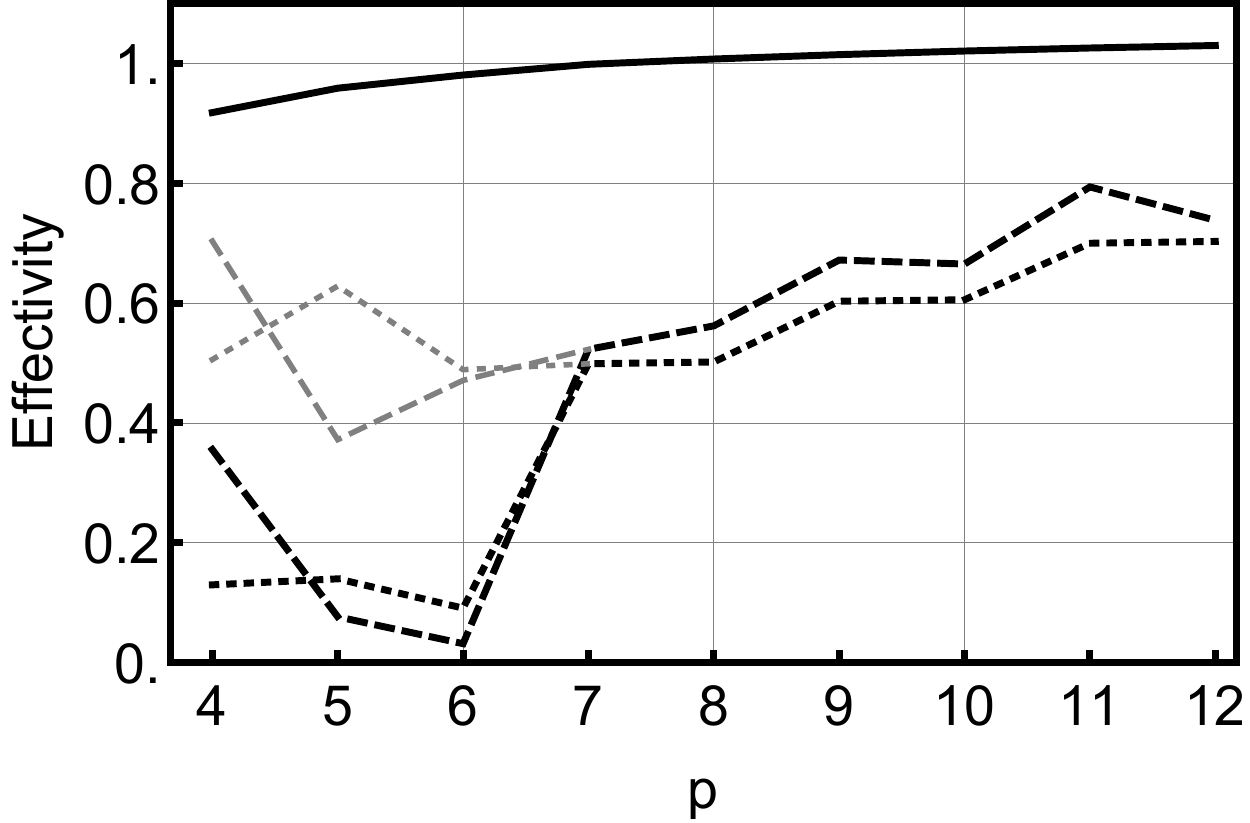}
	}\quad
	\subfloat[$hp$-version: Eigenvector  effectivity in $L^2$-norm.]
	{
	\includegraphics[width=0.45\textwidth]{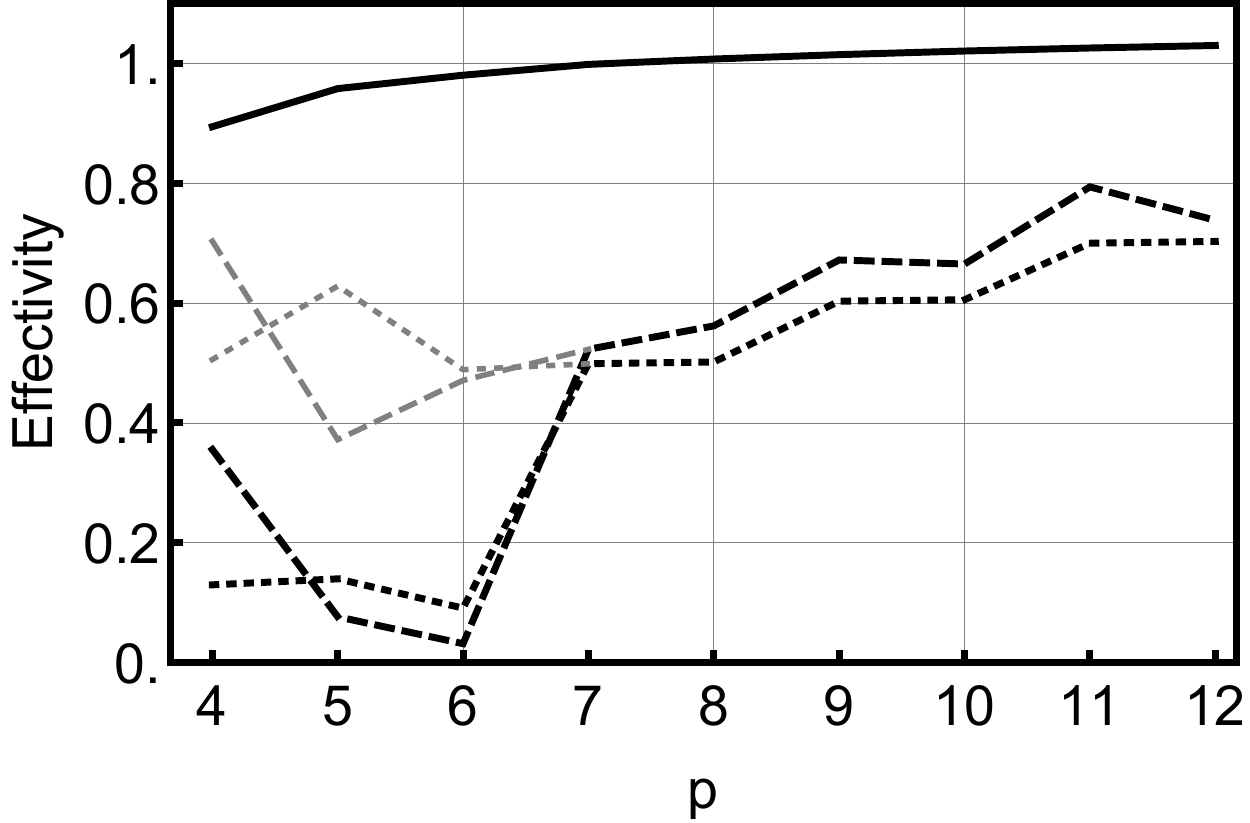}
	}\\
	\subfloat[$p$-version: Eigenvector  effectivity in $H^1$-seminorm.]
	{
	\includegraphics[width=0.45\textwidth]{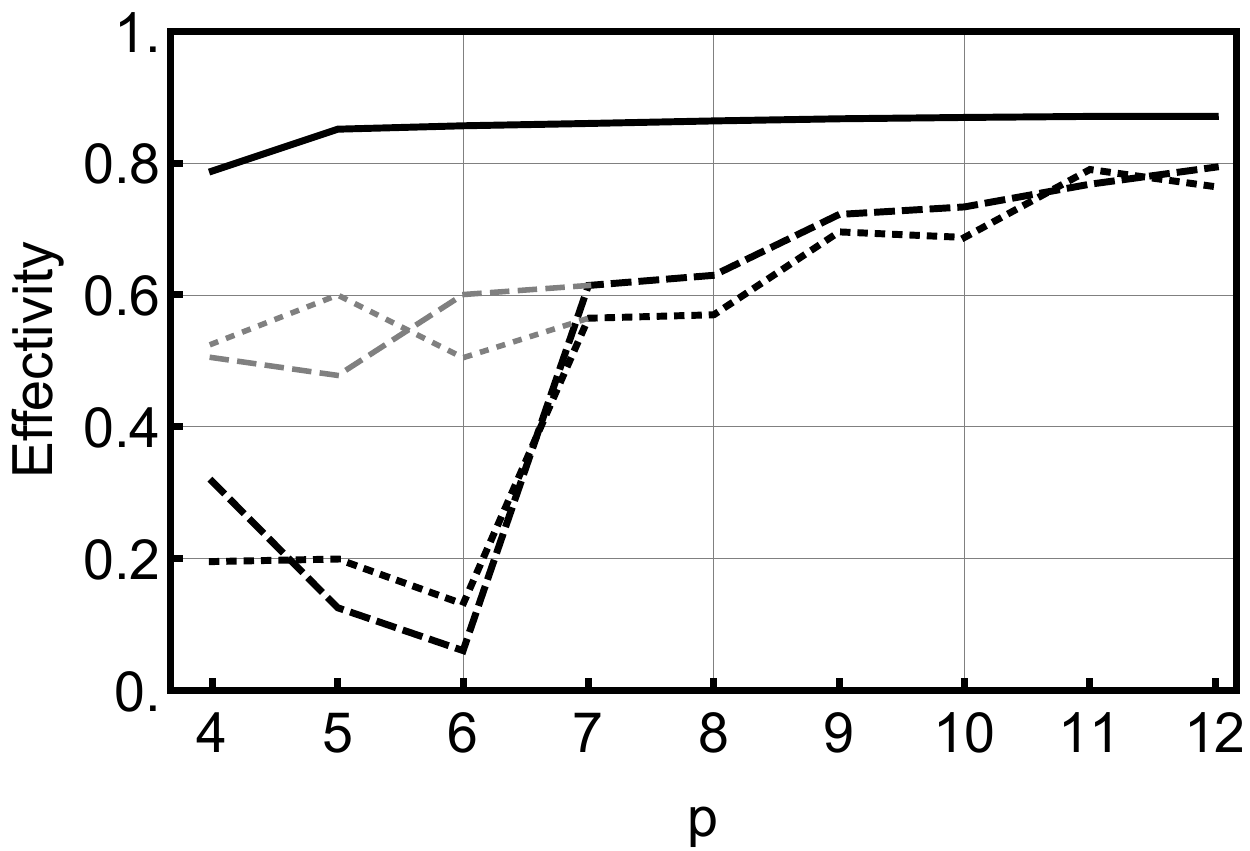}
	}\quad
	\subfloat[$hp$-version: Eigenvector  effectivity in $H^1$-seminorm.]
	{
	\includegraphics[width=0.45\textwidth]{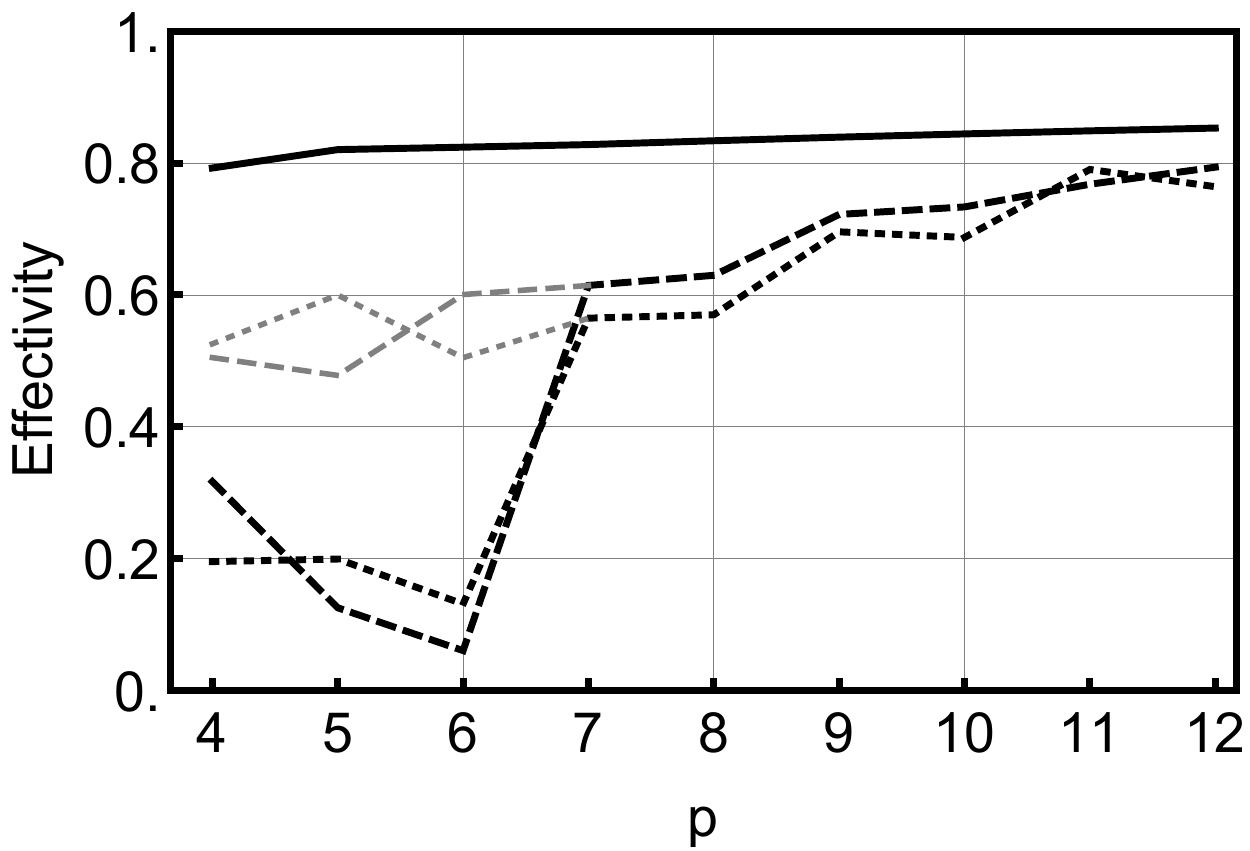}
	}
	\caption{Slit disk; Eigenvalue and eigenvector effectivity
          ratios for $(\lambda_1,\psi_1)$ (solid lines), $(\lambda_{52},\psi_{52})$
          (dashed lines), $(\lambda_{53},\psi_{53})$ (dotted lines),
          for both the $p$- and $hp$-version.  Gray curves correspond
          to the effectivities of the estimates when the computed
          eigenmode $\hat\psi_k$ is compared with the true eigenmode
          $\psi_j$ that it most closely resembles, and the computed
          eigenvalue $\hat\lambda_k$ is compared with the
          corresponding $\lambda_j$.}
	\label{fig:SlitDiskClusterPvsHP}	
\end{figure}
\begin{figure}
\centering
\subfloat[Eigenvalue error.]
{
\includegraphics[width=0.45\textwidth]{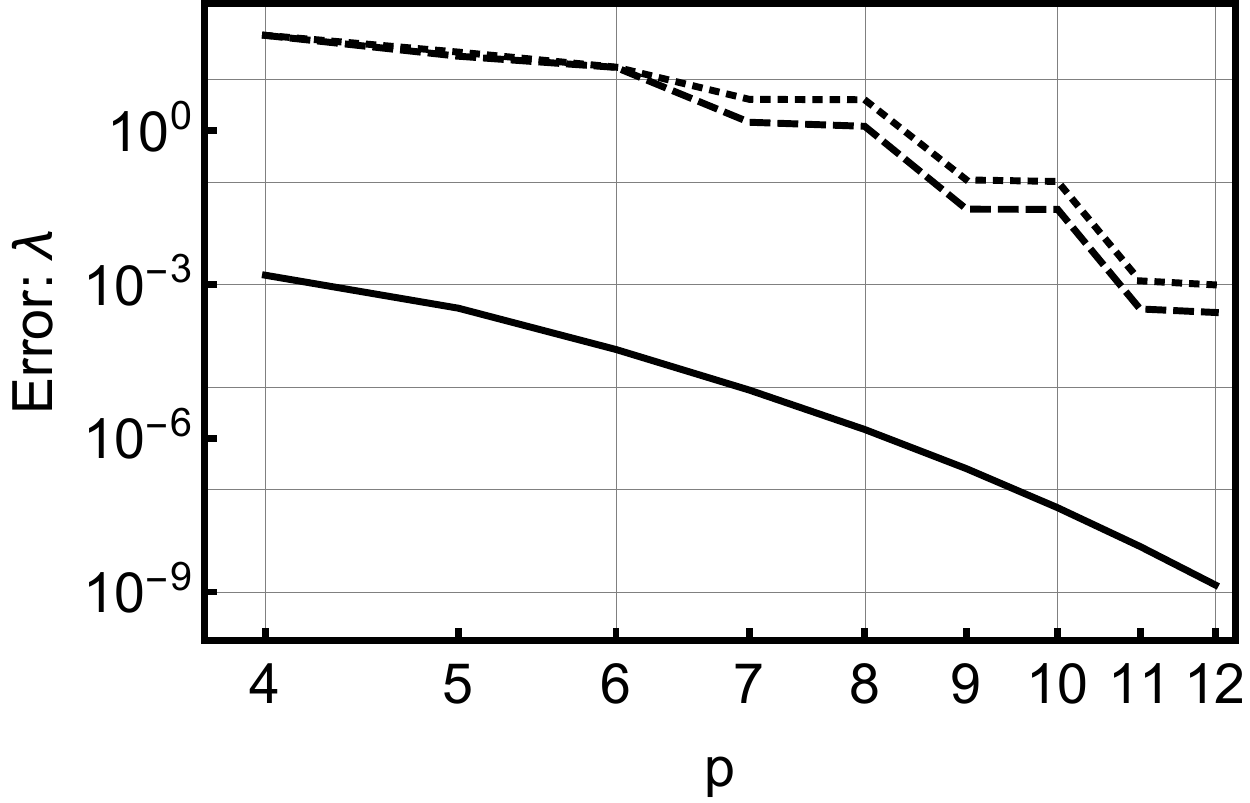}}\\
\subfloat[Eigenvector error in squared $L^2$-norm.]
{
\includegraphics[width=0.45\textwidth]{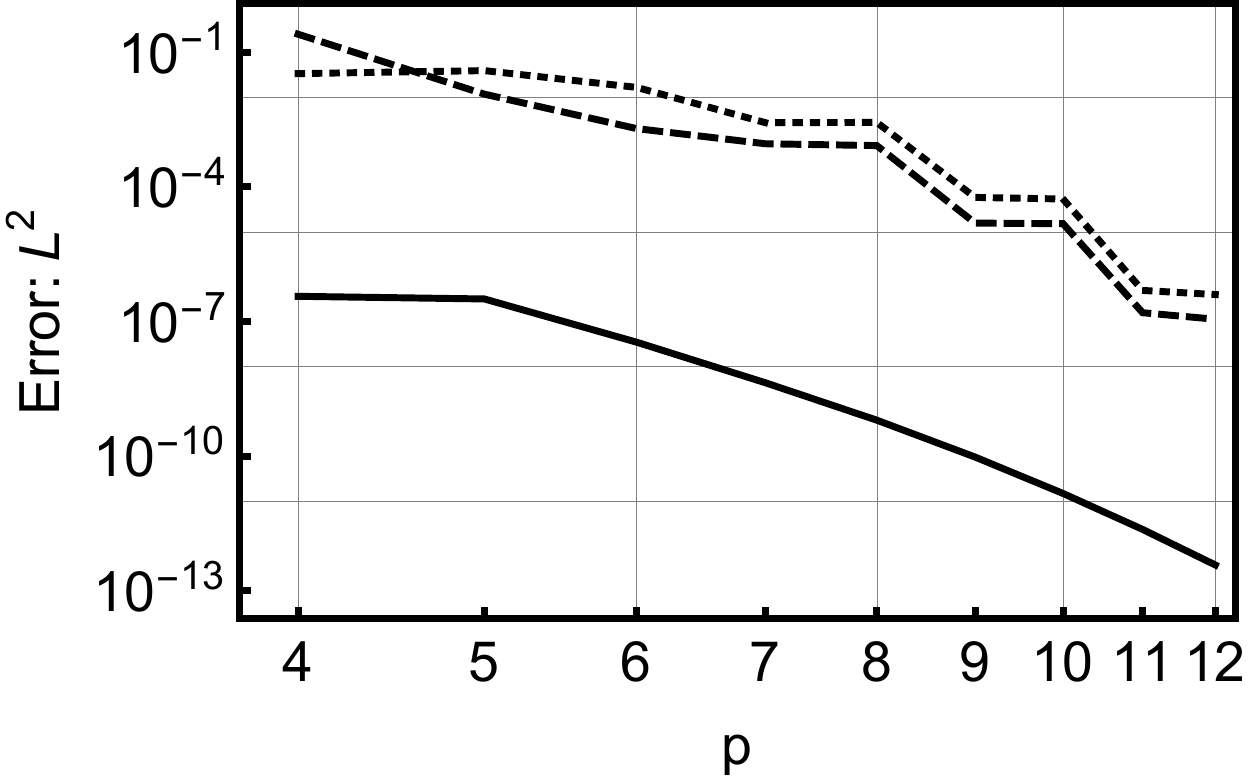}
}\quad
\subfloat[Eigenvector error in squared $H^1$-seminorm.]
{
\includegraphics[width=0.45\textwidth]{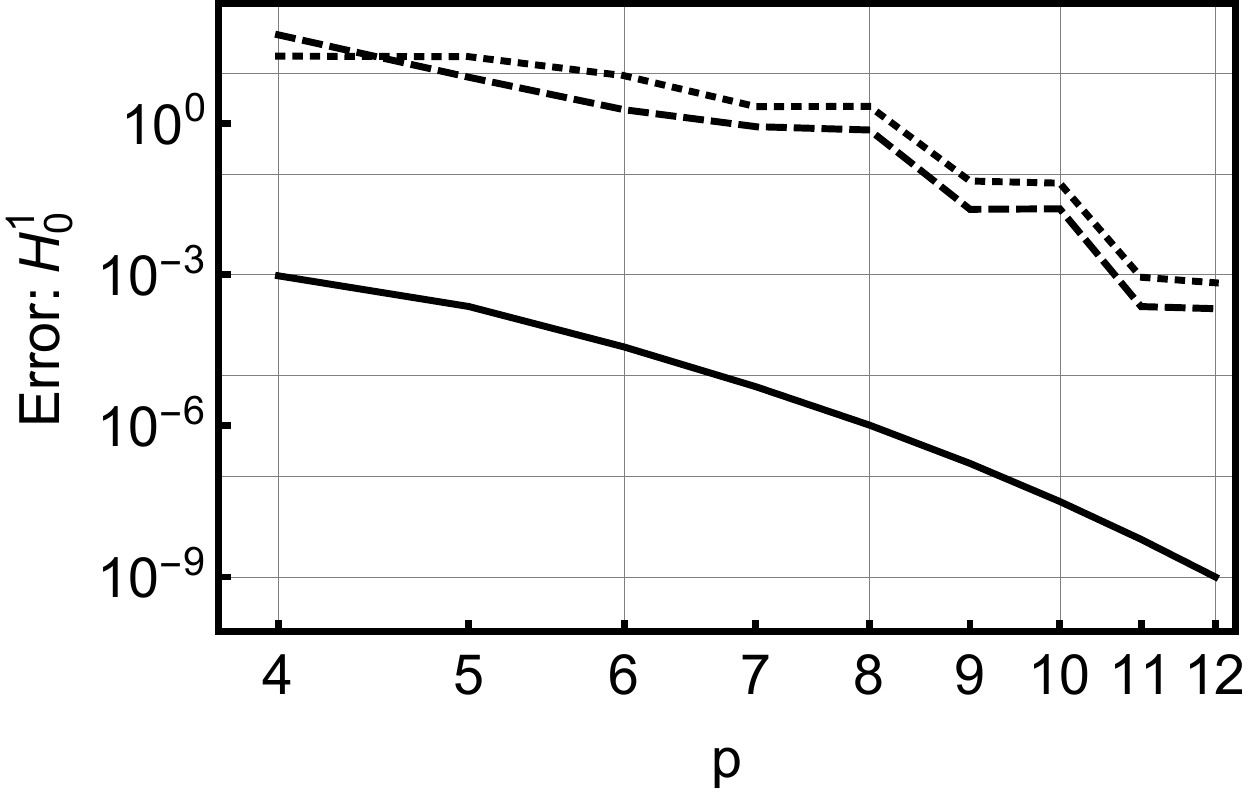}
}
\caption{Slit disk;  Eigenvalue and eigenvector convergence for $(\lambda_1,\psi_1)$ (solid lines), $(\lambda_{52},\psi_{52})$
		  (dashed lines), $(\lambda_{53},\psi_{53})$ (dotted lines)
		  for the $p$-version.
\label{fig:SlitDiskClusterPConvergence}}
\end{figure}

Before moving on to an empirical investigation of Proposition~\ref{BauerFike}, we provide
a few more remarks concerning the gray curves in Figure~\ref{fig:SlitDiskClusterPvsHP}.
Recalling~\eqref{EigenvectorErrors1}, we have that
$\enorm{\varepsilon(\hat\lambda_k\hat\psi_k)}$ approximates
$\enorm{\hat\psi_k-\cA^{-1}(\hat\lambda_k\hat\psi_k)}$ which, in turn,
approximates $\enorm{(I-S) \hat\psi_k}$, where $S$ is the spectral
projector for \textit{some} subset of eigenvalues $\Lambda$.  The
theory does not force any particular choice of $\Lambda$.  For
example, one may choose $\Lambda=\{\lambda_j\}$ for some $j\neq k$.
The consequences of different choices are reflected in the
``constant'' $C(\hat\lambda_k,\Lambda)$, which will blow up as
$\hat\lambda_k\to\lambda_k$ if $\lambda_k\not\in\Lambda$.
Informally, we can say that $\enorm{\varepsilon(\hat\lambda_k\hat\psi_k)}$ provides a
reasonable estimate of the
error $\enorm{\hat\psi_k-\psi_j}$ for \textit{some} eigenmode $\psi_j$
(after suitable normalization), but $\psi_j$ might not be an eigenmode
for $\lambda_k$ until the discretization is sufficiently rich.

Recalling the notation $K=\spec(H,G)$ and
$\tilde{K}=\spec (\tilde{H},G)$, we now illustrate the Bauer-Fike
estimate~\eqref{ApproxDef_Hausdorff}, using the matrix $2$-norm for
the upper bound.  Note that
  $\|G^{-1/2}\|_2^2=(\hat\mu_1)^{-1}$, where $\hat\mu_1$ is the
  smallest approximate eigenvalue for the cluster of interest.  More
  specifically, we will empirically compare both sides of the
  inequality
\begin{align}\label{BauerFikeExperiments}
\mathrm{dist}(K,\tilde{K})\leq \|H-\tilde{H}\|_2/\hat\mu_1~,
\end{align}
i.e. we compare the quantities
$\mathrm{dist}(K,\tilde{K})$ and $\|H-\tilde{H}\|_2/\hat\mu_1$.
The behavior of the Hausdorff distance $\mathrm{dist}(K,\tilde{K})$
demonstrates that our computable $\tilde{H}$ provides a spectrally
accurate approximation of $H$, and is therefore suitable for more
nuanced estimates than those of trace-type
(e.g.~\eqref{TraceEstimates3a}).  Comparison of the both sides of the
inequality indicates that the norm bound is not a gross overestimate,
and may in fact provide a relatively tight bound.  

We first consider the scenario in which the cluster of interest is
fixed, namely $\{\lambda_{52},\lambda_{53}\}$, and we observe the
behavior of both sides of~\eqref{BauerFikeExperiments} as the
discretization parameter $p$ is increased.  In this case, $\hat\mu_1$
decreases toward $\lambda_{52}$ as $p$ increases.  The results of
these experiments are summarized in Figure~\ref{fig:BauerFikeCluster}.
\begin{figure}
	\centering
	\subfloat[$\mathrm{dist}(K,\tilde{K})$ (solid) and $\|H-\tilde{H}\|_2/\hat\mu_1$
        (dashed); visually indistinguishable.]
	{
	\includegraphics[width=0.45\textwidth]{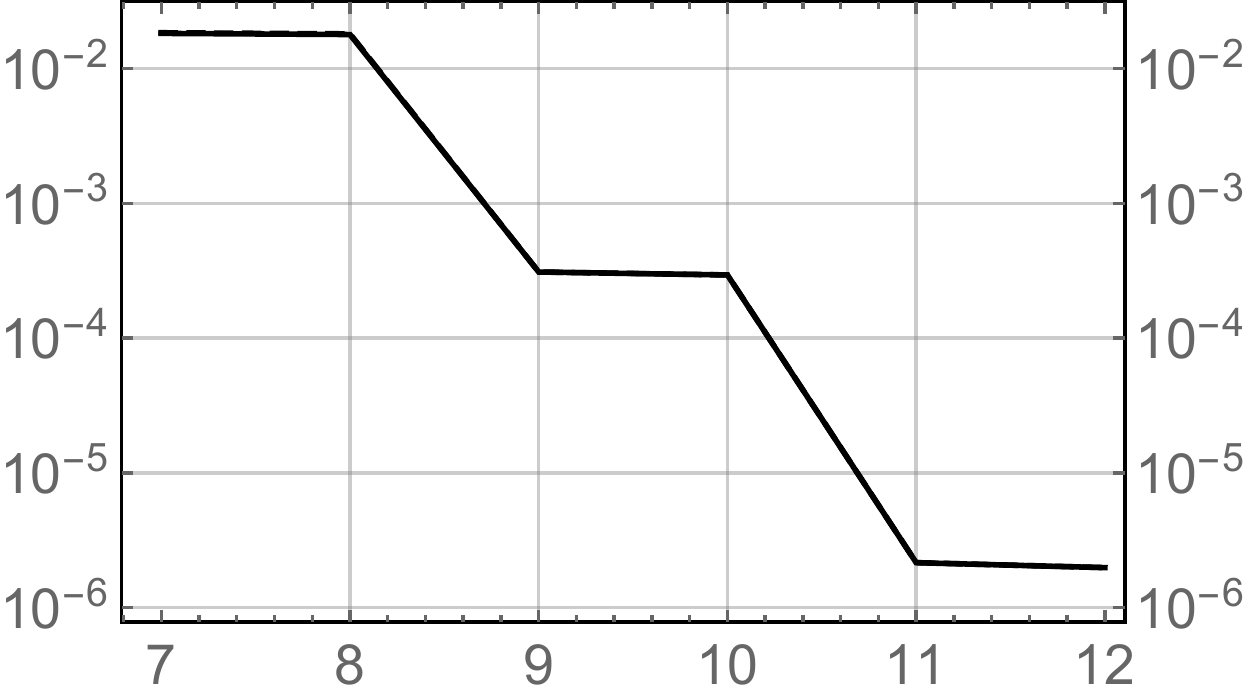}
	}\quad
	\subfloat[%
	Relative error $|a-b|/b$, where
          $a=\mathrm{dist(K,\tilde{K})}$ and $b=\|H-\tilde{H}\|_2/\hat\mu_1$.%
	]
	{
	\includegraphics[width=0.45\textwidth]{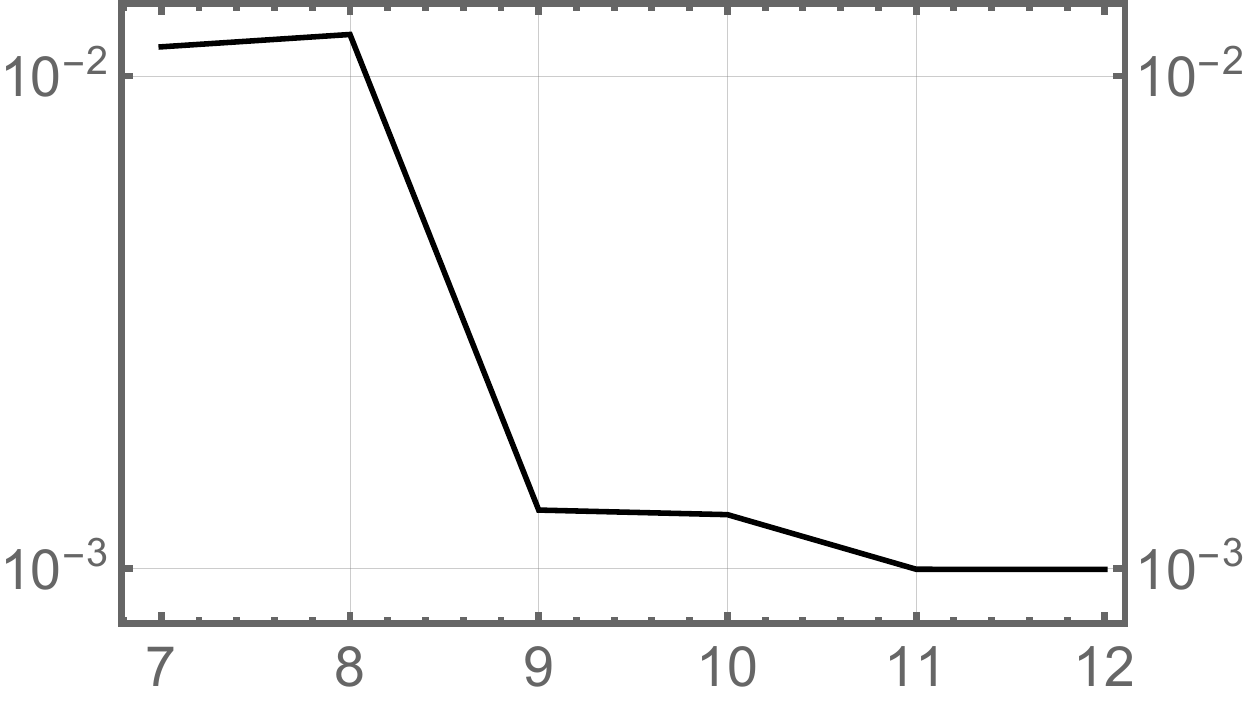}
	}
	\caption{Illustrating~\eqref{BauerFikeExperiments} for the Slit Disk eigenvalue cluster
          $\{\lambda_{52}, \lambda_{53}\}$, with
          $7\leq p\leq 12$.}

	\label{fig:BauerFikeCluster}
\end{figure}
We observe the same stairstep convergence of both quantities as
before, and note that $\|H-\tilde{H}\|_2/\hat\mu_1$ provides a very
tight upper bound on $\mathrm{dist}(K,\tilde{K})$ in this case.

We next consider the scenario in which the discretization parameter
$p$ is fixed, and the size of cluster is increased.  More
specifically, we consider two fixed discretizations, with $p=7$ and
$p=12$, and investigate both sides of~\eqref{BauerFikeExperiments}
as the cluster of interest, $\{\lambda_1,\dots,\lambda_r\}$, grows
with $r$, $1\leq r\leq 60$.  In this case,
$\hat\mu_1\approx\lambda_1$ is fixed as $r$ varies, so we expect the
upper bound $\|H-\tilde{H}\|_2/\hat\mu_1$ to become more pessimistic
as $r$ increases.  This expectation is confirmed in
Figure~\ref{fig:BauerFike}, where we nonetheless observe that both
quantities exhibit very similar qualitative behavior.
\begin{figure}
	\centering
	\subfloat[Both sides of the estimate at $p=7$.]
	{
	\includegraphics[width=0.45\textwidth]{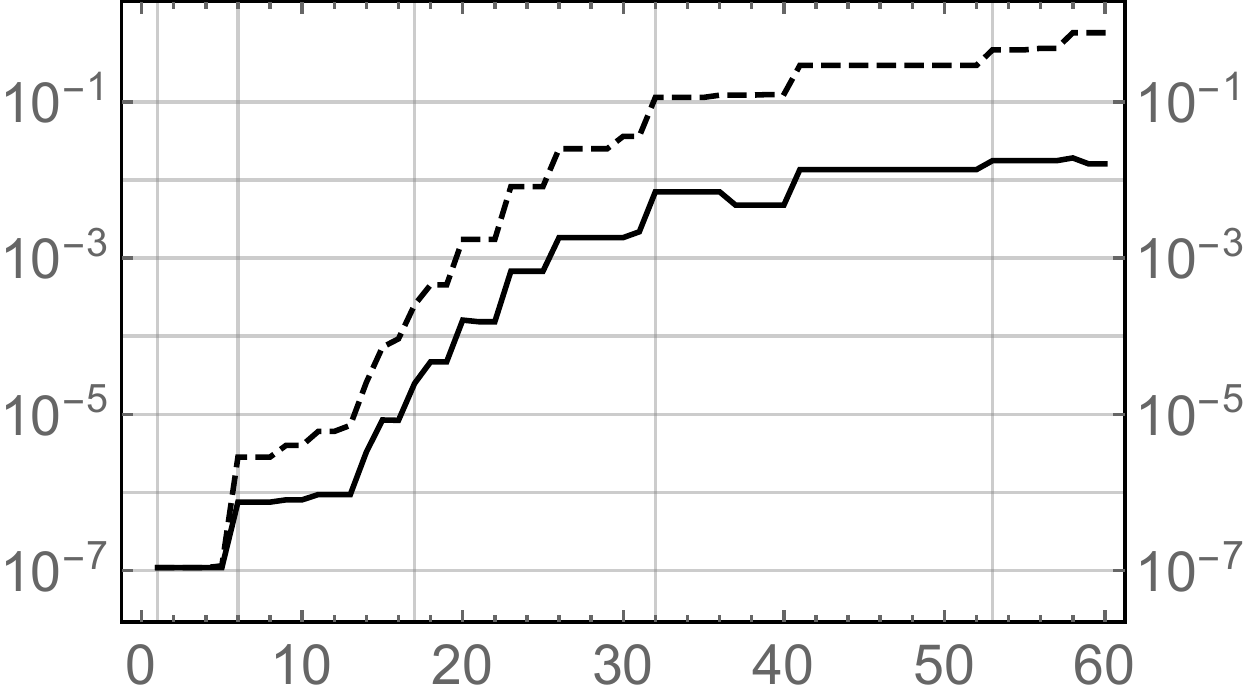}
	}\quad
	\subfloat[Both sides of the estimate at $p=12$.]
	{
	\includegraphics[width=0.45\textwidth]{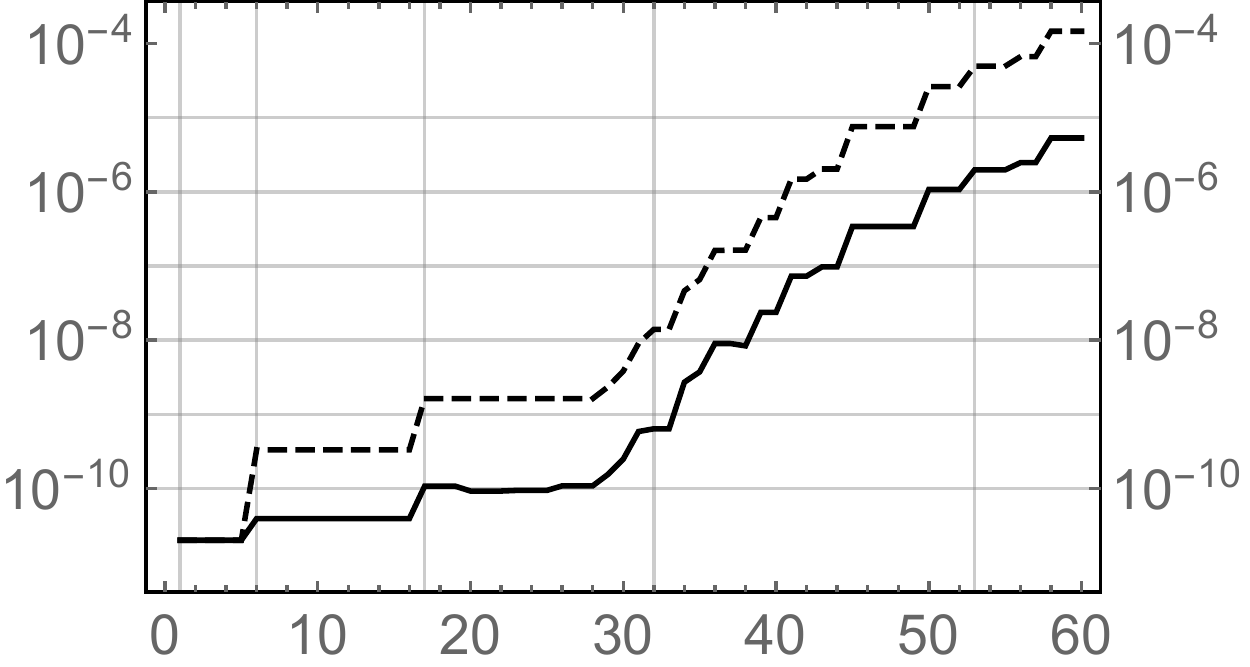}
	}
	\caption{\label{fig:BauerFike}Illustrating~\eqref{BauerFikeExperiments} for the Slit Disk eigenvalue cluster
          $\{\lambda_{1}, \ldots,\lambda_{j}\}$, $j=1,\ldots,60$, for
          $p=7$ (left) and $p=12$; $\mathrm{dist}(K,\tilde{K})$ (solid) and $\|H-\tilde{H}\|_2/\hat\mu_1$
        (dashed).}
\end{figure}
Without going so far as to make a conjecture, we note the correlation
between the more significant jumps in these graphs and the
inclusion in the cluster of interest of the eigenfunctions having the strongest singularities,
$\psi\sim r^{1/2}$ as $r\to 0$, namely $\{\psi_1,\psi_6,\psi_{17},\psi_{32},\psi_{53}\}$.
\end{example}

\begin{remark}[Mode Detection]
  The exact eigenmodes (\ref{eq:exactmodes}) have a tensor product
  structure.  This simplifies greatly the task of identifying the
  closest mode $\psi_{m,n}$ to some computed $\hat\psi$. The indices
  $m$ and $n$ represent the radial and angular parts, respectively, of
  $\psi_{m,n}$ and thus the mode detection approach is to find $m$ and
  $n$ that best correspond to the computed eigenmode. For identifying
  the angular part $m$, $\hat\psi$ is evaluated along circles at two
  randomly chosen radii $r_1$ and $r_2$, and the wave number along
  these circles is computed using the discrete Fourier transform (DFT).  In
  the unlikely case of the two values being different, a third radius
  is chosen for tie-breaking. We proceed similarly for the radial
  direction. However, in the absence of equivalent to the DFT, we
  project onto a set of admissible radial profiles and choose the one
  that is closest in the $L^2$ sense.
\end{remark} 
\section{A Posteriori Estimates for Clusters: Numerical Experiments}\label{sec:Exper}

In this section the focus is on a set of problems where the spectrum
has a structure rich in clusters that can be identified a priori with
high confidence. In this setting, it is best to estimate eigenvalue
error and associated invariant subspace error over the clusters either
with trace estimates such
as~\eqref{TraceEstimates4a}-\eqref{TraceEstimates4b} or via looking
directly at $\spec(\tilde{H},G)$.  As a starting point, we consider a
pair of complementary problems posed on half-disks, first studied by
Jacobson et al. ~\cite{Jakobson2006141}, where they were shown to have
identical spectra.  We then consider two sets of configurations
derived from the original pair by connecting these half-disks with
narrow bridges, see Figure~\ref{fig:BridgeDomain}.  This family of
configurations is such that pairs of nearby eigenvalues are expected
around each of the eigenvalues of the isospectral problems.

\subsection{Isospectral Problems}\label{sec:Reference Domain}
Let $\Omega=\{(x,y):\,x^2+y^2<1\,,\,y>1\}$
be the half-disk,
with boundary $\partial\Omega$ split into four parts,
$\partial\Omega=\gamma_1\cup\gamma_2\cup\gamma_3\cup\gamma_4$, where
\begin{align*}
\gamma_1&=\{(r\cos\theta,r\sin\theta):\;\theta=0\,,\,0\leq r\leq1\}\cup
\{(r\cos\theta,r\sin\theta):\;r=1\,,\,0\leq\theta\leq \pi/4\}~,\\
\gamma_2&=\{(r\cos\theta,r\sin\theta):\;r=1\,,\,\pi/4\leq\theta\leq 3\pi/4\}~,\\
\gamma_3&=\{(r\cos\theta,r\sin\theta):\;r=1\,,\,3\pi/4\leq\theta\leq \pi\}~,\\
\gamma_4&=\{(r\cos\theta,r\sin\theta):\;\theta=\pi\,,\,0\leq r\leq1\}~.
\end{align*}
The domain and boundary decomposition are shown in
Figures~\ref{fig:A Domain} and~\ref{fig:B Domain}.  
We consider a pair of complementary
problems in which we alternately apply Dirichlet and Neumann
conditions on the even and odd parts of the boundary,
\begin{align}
-\Delta \psi = \lambda\psi\quad,\quad \psi=0\mbox{ on }
  \gamma_1\cup\gamma_3\quad,\quad
\partial\psi/\partial n = 0\mbox{ on }\gamma_2\cup\gamma_4~,\label{Jacobson1}\\
-\Delta \psi = \lambda\psi\quad,\quad \psi=0\mbox{ on }
  \gamma_2\cup\gamma_4\quad,\quad
\partial\psi/\partial n = 0\mbox{ on }\gamma_1\cup\gamma_3~. \label{Jacobson2}
\end{align}
As was proved in~\cite{Jakobson2006141}, these problems are
\textit{isospectral}.  In other words, the eigenvalues
of~\eqref{Jacobson1} are identical to those of~\eqref{Jacobson2}.
In Figure~\ref{fig:Isospectral}, we show a few eigenvectors associated
with both problems.  These were computed using refinement strategy
that ensures that the lower part of the spectrum is accurately
resolved.  Reference values for the first fifteen eigenvalues are
given in Table~\ref{tab:Isospectral}.
\begin{figure}
	\centering
	\subfloat[Problem~\eqref{Jacobson1}]
	{\label{fig:A Domain}\includegraphics[width=0.45\textwidth]{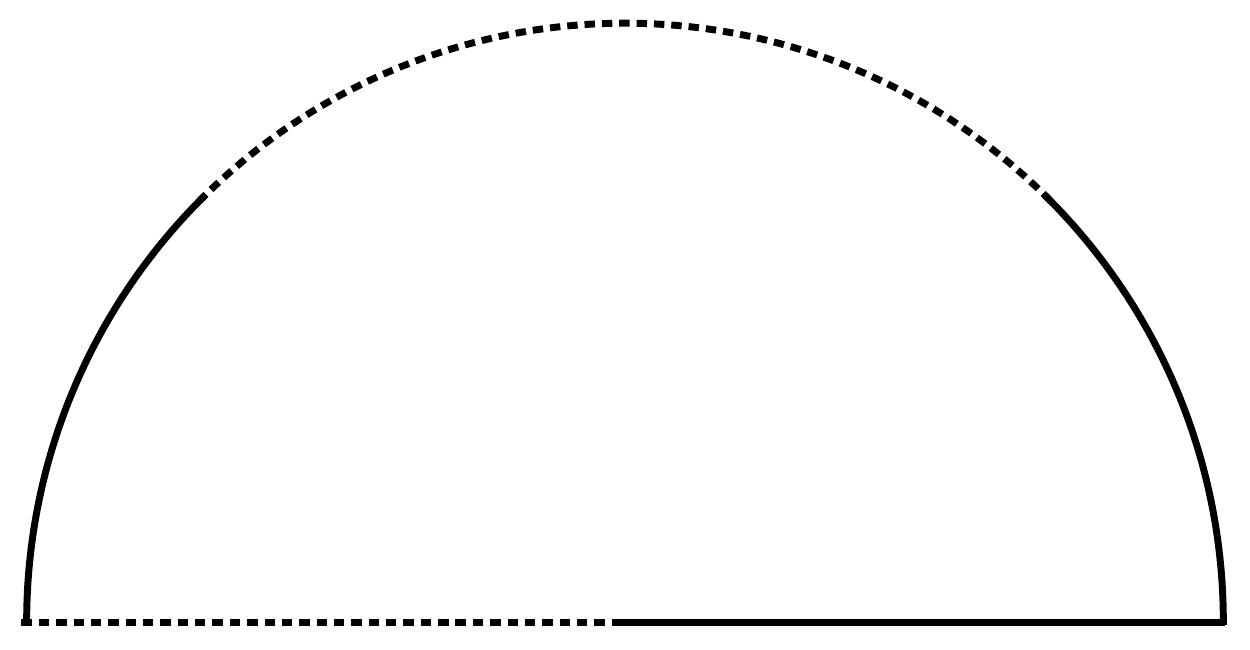}}\quad
	\subfloat[Problem~\eqref{Jacobson2}]
	{\label{fig:B
            Domain}\includegraphics[width=0.45\textwidth]{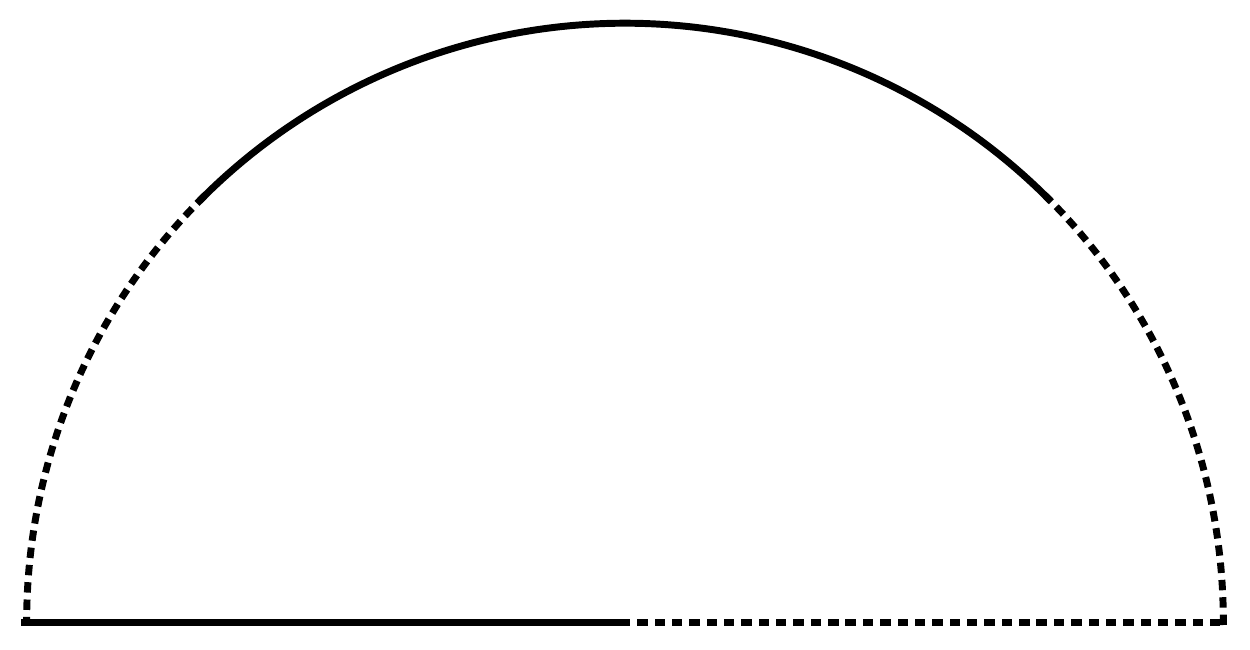}}\\
	\subfloat[$\psi_1$ for Problem~\eqref{Jacobson1}]
	{\label{fig:A Modes 1}\includegraphics[width=0.45\textwidth]{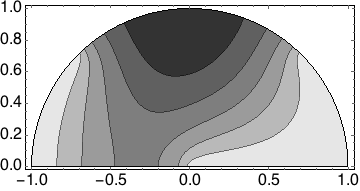}}\quad
	\subfloat[$\psi_1$ for Problem~\eqref{Jacobson2}]
	{\label{fig:B Modes 1}\includegraphics[width=0.45\textwidth]{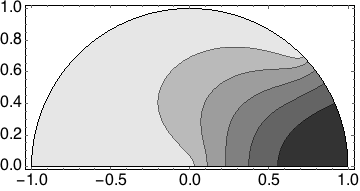}}\\
	\subfloat[$\psi_5$ for Problem~\eqref{Jacobson1}]
	{\label{fig:A Modes 5}\includegraphics[width=0.45\textwidth]{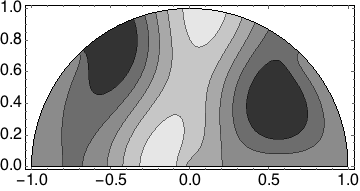}}\quad
	\subfloat[$\psi_5$ for Problem~\eqref{Jacobson2}]
	{\label{fig:B Modes 5}\includegraphics[width=0.45\textwidth]{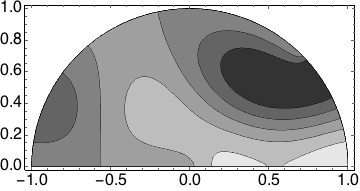}}\\
	\subfloat[$\psi_{11}$ for Problem~\eqref{Jacobson1}]
	{\label{fig:A Modes 11}\includegraphics[width=0.45\textwidth]{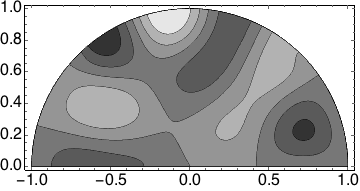}}\quad
	\subfloat[$\psi_{11}$ for Problem~\eqref{Jacobson2}]
	{\label{fig:B Modes 11}\includegraphics[width=0.45\textwidth]{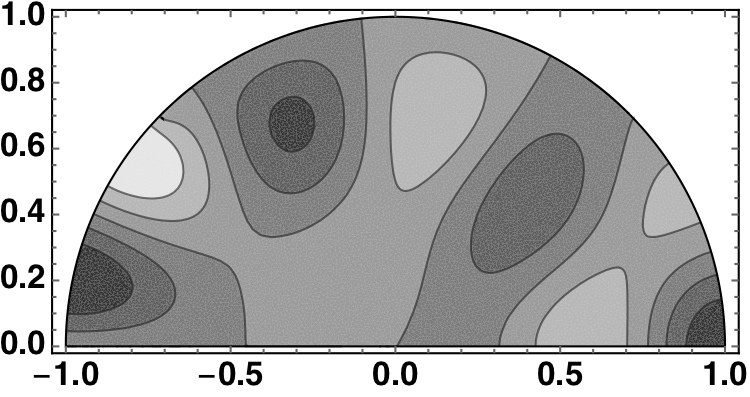}}
	\caption{Isospectral problems. The components of the boundary having Dirichlet and Neumann
	 conditions are drawn with solid and dotted lines, 
	 respectively, in (a) and (b).}\label{fig:Isospectral}
\end{figure}

\begin{table}
\caption{Isospectral problems.  Reference values for the lowest
  $15$ eigenvalues for
  problems~\eqref{Jacobson1}-\eqref{Jacobson2}.\label{tab:Isospectral}}
\centering
	\begin{tabular}{cl|cl|cl}
		$i$ & \multicolumn{1}{c}{$\lambda_i$}&$i$ & \multicolumn{1}{c}{$\lambda_i$} &$i$ & \multicolumn{1}{c}{$\lambda_i$} \\ \hline
		1 & 4.50351270364&6 & 4.63221446587$\times 10^1$ & 11 &  8.34387148427$\times 10^1$\\
		2 & 1.35208410401$\times 10^1$&7 & 5.13074786442$\times 10^1$ & 12 & 9.11669451784$\times 10^1$\\
		3 & 1.98639263212$\times 10^1$&8 & 6.24572729970$\times 10^1$ & 13 & 1.04631385585$\times 10^2$\\
		4 & 3.04933490983$\times 10^1$&9 & 6.74067396593$\times 10^1$ & 14 & 1.09930498884$\times 10^2$\\
		5 & 3.51893179474$\times 10^1$&10 & 7.87626319950$\times 10^1$ & 15 & 1.11846648035$\times 10^2$\\ \hline	
	\end{tabular}
\end{table}

\subsection{Bridge Configurations}\label{sec:Bridge
  Configurations}
By joining two of the isospectral drums above with a narrow bridge, we
can create a family of configurations in which there are clusters of
eigenvalues throughout the spectrum near predictable numbers,
i.e. near the eigenvalues of the isospectral domains.  We take the
domain to be two half-disks of raduis $1$ connected by a
$1/10\times 1/4$ rectangular bridge, see
Figure~\ref{fig:BridgeDomainLabeled}.  In this figure, we have labeled
segments of the boundary A-L, and we obtain different configurations
by assigning either homogeneous Dirichlet or Neumann conditions to
these edges.  Taking both sides of the bridge to have the same type of
boundary condition, either both Dirichlet or both Neumann, there are
20 such configurations that are associated with the isospectral pair
from Section~\ref{sec:Reference Domain}, 10 having the Dirichlet
bridge and 10 having the Neumann bridge.  These are tabulated in
Table~\ref{tbl:BridgeDomainConfigurations}, and two such
configurations are shown in Figures~\ref{fig:BridgeDomainD2}
and~\ref{fig:BridgeDomainN9}.  For each eigenvalue
of~\eqref{Jacobson1}-\eqref{Jacobson2}, we expect to have a pair of
eigenvalues on the Bridge domain that are close to it, regardless of
which of the 20 configurations of boundary conditions that we use.
This is illustrated in Table~\ref{tbl:BridgeDomainEigenvalues}, where
we give reference values for the first 12 eigenvalues of the
configurations pictured in Figure~\ref{fig:BridgeDomain}, together
with the first 6 eigenvalues of the isospectral domains for
comparison.  Contour plots of the ninth and tenth eigenvectors for
both of these configurations are given in
Figure~\ref{fig:BridgeDomainEigenvectors}.  As above, we employ
refinement strategies for our experiments that ensure that the lower
part of the spectrum is accurately resolved and the observed phenomena
are not simply artifacts of the discretization, see
Figure~\ref{fig:BridgeDomainMesh}.
\begin{figure}
\centering
\subfloat[Bridge domain with labeled segments.]
{\label{fig:BridgeDomainLabeled}
\includegraphics[width=0.45\textwidth]{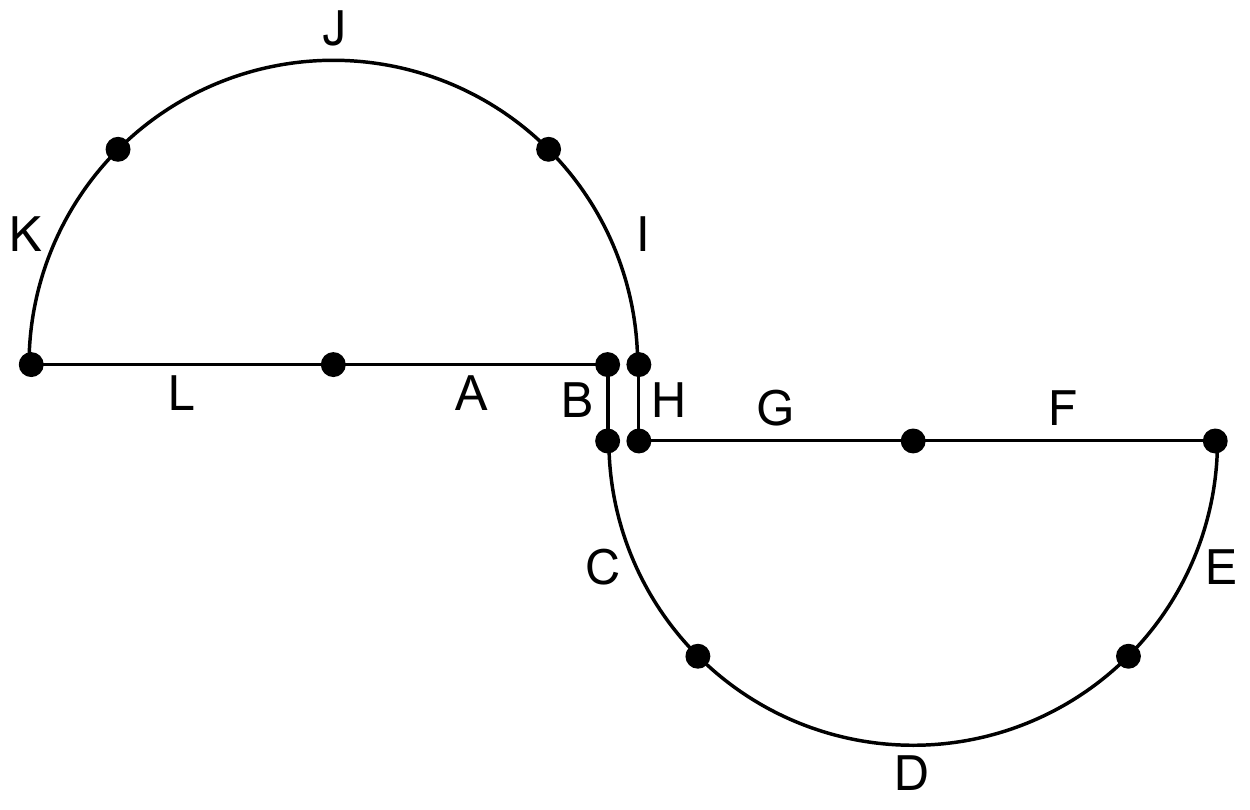}}\quad
\subfloat[$p$-type mesh used in experiments.]
{\label{fig:BridgeDomainMesh}
\includegraphics[width=0.45\textwidth]{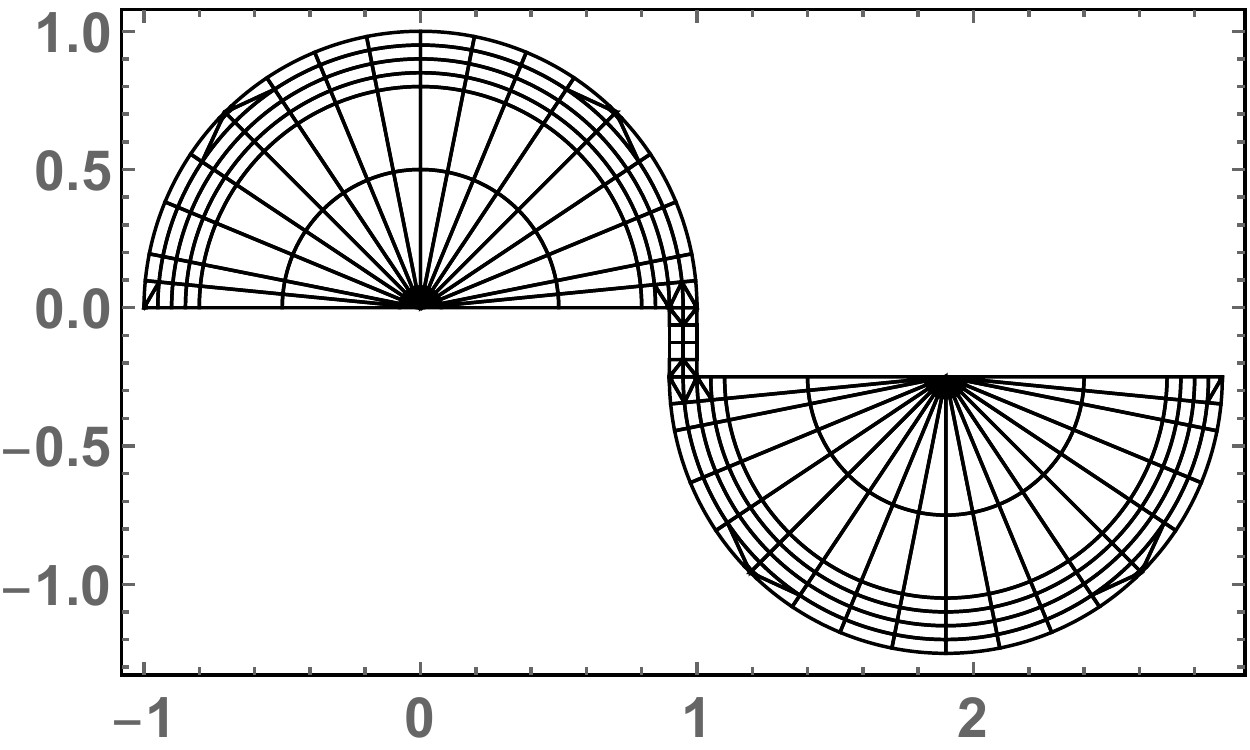}}\\
\subfloat[Dirichlet bridge, Case 2.]
{\label{fig:BridgeDomainD2}
\includegraphics[width=0.45\textwidth]{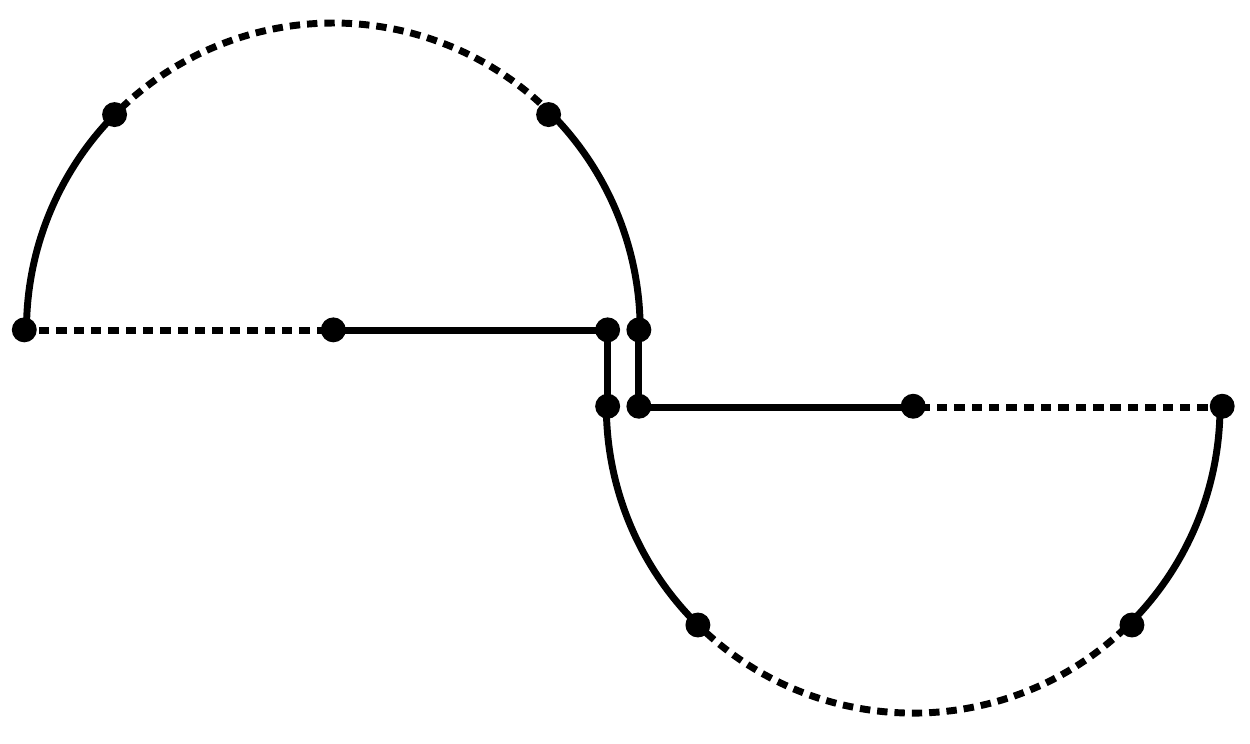}}\quad
\subfloat[Neumann bridge, Case 9.]
{\label{fig:BridgeDomainN9}
\includegraphics[width=0.45\textwidth]{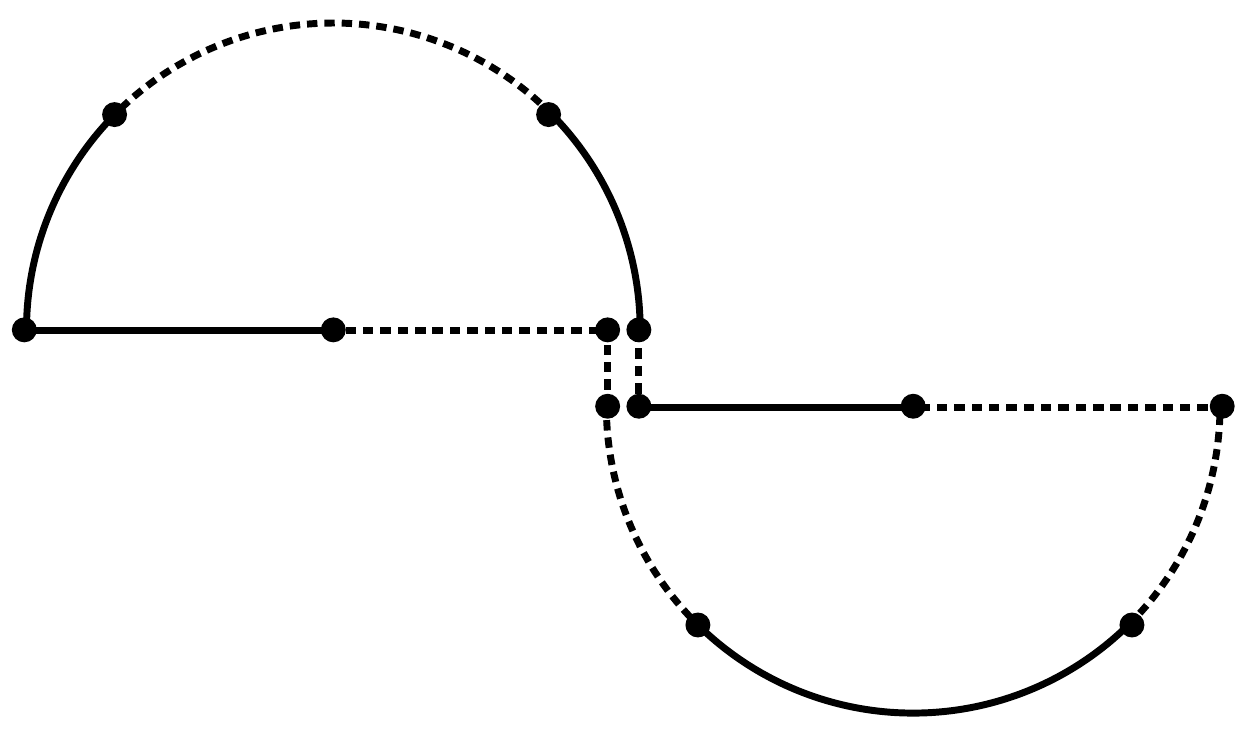}}
\caption{Bridge domain, a computational mesh and two 
  configurations. \label{fig:BridgeDomain}}
\end{figure}
\begin{table}
  \centering
  \begin{tabular}{ccccccccccccc}
    Case & \text{A} & \text{B} & \text{C} & \text{D} & \text{E} & \text{F} & \text{G} & \text{H} & \text{I} & \text{J} & \text{K} & \text{L} \\ \hline
    1 &	 \text{D} & \text{X} & \text{D} & \text{N} & \text{D} & \text{D} & \text{N} & \text{X} &
                                                                                                 \text{D} & \text{N} & \text{D} & \text{N} \\
    2 &	 \text{D} & \text{X} & \text{D} & \text{N} & \text{D} & \text{N} & \text{D} & \text{X} &
                                                                                                 \text{D} & \text{N} & \text{D} & \text{N} \\
    3 &	 \text{D} & \text{X} & \text{N} & \text{D} & \text{N} & \text{N} & \text{D} & \text{X} &
                                                                                                 \text{D} & \text{N} & \text{D} & \text{N} \\
    4 &	 \text{D} & \text{X} & \text{N} & \text{D} & \text{N} & \text{D} & \text{N} & \text{X} &
                                                                                                 \text{D} & \text{N} & \text{D} & \text{N} \\
    5 &	 \text{N} & \text{X} & \text{D} & \text{N} & \text{D} & \text{D} & \text{N} & \text{X} &
                                                                                                 \text{N} & \text{D} & \text{N} & \text{D} \\
    6 &	 \text{N} & \text{X} & \text{N} & \text{D} & \text{N} & \text{N} & \text{D} & \text{X} &
                                                                                                 \text{N} & \text{D} & \text{N} & \text{D} \\
    7 &	 \text{N} & \text{X} & \text{N} & \text{D} & \text{N} & \text{D} & \text{N} & \text{X} &
                                                                                                 \text{N} & \text{D} & \text{N} & \text{D} \\
    8 &	 \text{N} & \text{X} & \text{D} & \text{N} & \text{D} & \text{D} & \text{N} & \text{X} &
                                                                                                 \text{D} & \text{N} & \text{D} & \text{D} \\
    9 &	 \text{N} & \text{X} & \text{N} & \text{D} & \text{N} & \text{N} & \text{D} & \text{X} &
                                                                                                 \text{D} & \text{N} & \text{D} & \text{D} \\
    10 &	 \text{D} & \text{X} & \text{N} & \text{D} & \text{N} & \text{N} & \text{D} & \text{X} &
                                                                                                         \text{N} & \text{D} & \text{N} & \text{N} \\ \hline
  \end{tabular}
\caption{Bridge domain configurations.  Edges marked D correspond to
  Dirichlet conditions, and those marked $N$ correspond to Neumann
  conditions. The Dirichlet bridge
  configurations correspond to X$=$D, and the Dirichlet bridge
  configurations correspond to X$=$N}\label{tbl:BridgeDomainConfigurations}
\end{table}

\begin{table}
\begin{center}
\begin{tabular}{|cc|cc|cc|}\hline
\multicolumn{2}{|c|}{Isospectral}&\multicolumn{2}{c|}{Dirchlet Case
                                   2}&\multicolumn{2}{c|}{Neumann Case
                                       9}\\
$i$&$\lambda_i$&$i$&$\lambda_i$&$i$&$\lambda_i$\\\hline
1& 4.50351270364&1&4.50348976806&1&4.50318419853\\
  &                          &2&4.50348977820&2&4.50836662912\\\hline
2& 13.5208410401&3&13.5207888798&3&13.4263953994\\
  &                          &4&13.5207889083&4&13.5657193361\\\hline
3& 19.8639263212&5&19.8636968115&5&19.5509676421\\
  &                          &6&19.8636969659&6&19.8768798947\\\hline
4& 30.4933490983&7&30.4931397957&7&30.2012278561\\
  &                          &8&30.4931399453&8&30.5972353211\\\hline
5& 35.1893179474&9&35.1878233714&9&35.0596433246\\
  &                          &10&35.1878245124&10&35.2057946583\\\hline 
6& 46.3221446587&11&46.3208464060&11&45.7623966583\\ 
  &                          &12&46.3208474584&12&46.4364126764\\\hline
\end{tabular}
\end{center}
\caption{\label{tbl:BridgeDomainEigenvalues} The lowest 12 eigenvalues for
two of the Bridge domain configurations compared with the lowest 6
eigenvalues for the isospectral domains.}
\end{table}
\begin{figure}
\centering
\subfloat[$\lambda_9=35.1878233714$.]
{\label{fig:BridgeDomainD2_9}
\includegraphics[width=0.45\textwidth]{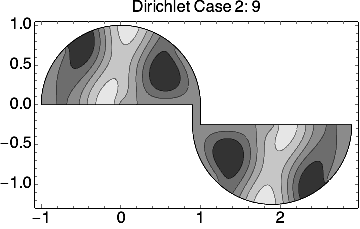}}\quad
\subfloat[$\lambda_{10}=35.1878245124$.]
{\label{fig:BridgeDomainD2_10}
\includegraphics[width=0.45\textwidth]{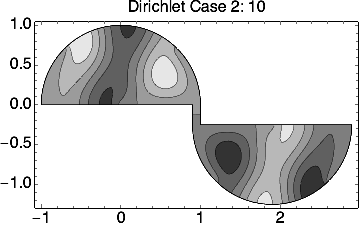}}\\
\subfloat[$\lambda_9=35.0596433246$.]
{\label{fig:BridgeDomainN9_9}
\includegraphics[width=0.45\textwidth]{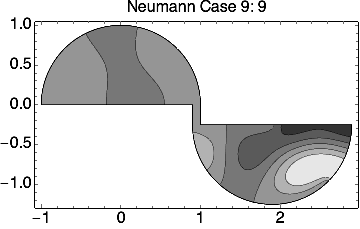}}\quad
\subfloat[$\lambda_{10}=35.2057946583$.]
{\label{fig:BridgeDomainN9_10}
\includegraphics[width=0.45\textwidth]{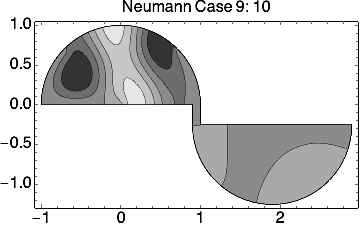}}
\caption{\label{fig:BridgeDomainEigenvectors} Contour plots of $\psi_9$
and $\psi_{10}$ for Case 2 of the Dirichlet bridge (top), and Case 9
of the Neumann Bridge (bottom).  Compare with Figures~\ref{fig:A Modes
  5} and~\ref{fig:B Modes 5}.  Reference eigenvalues are given with
each plot.}
\end{figure}

For our first set of experiments with these two configurations, we
consider the Hausdorff distance, $\mathrm{dist}(\Lambda,\hat\Lambda)$, between the reference eigenvalues
$\Lambda=\{\lambda_1,\ldots,\lambda_r\}$ and the computed eigenvalues
$\hat\Lambda=\{\hat\lambda_1,\ldots,\hat\lambda_r\}$ over a range of
discretizations, for different values of $1\leq r\leq 12$.  More
specifically, we compare this Hausdorff distance with our a posteriori
error estimate of it,
\begin{align}\label{EigenvalueClusterHausdorffEst}
\mathrm{dist}(\Lambda,\hat\Lambda)\doteq \max\left\{\max_{\lambda\in \Lambda}\min_{\hat\lambda\in\hat{\Lambda}}|\lambda-\hat\lambda|\,,\,
\max_{\hat\lambda\in\hat{\Lambda}}\min_{\lambda\in
  \Lambda}|\lambda-\hat\lambda|\right\}\approx \lambda_{\max}(\tilde{H})~,
\end{align}
where $\tilde{H}\in\RR^{r\times r}$ is given in~\eqref{HApprox}.  
This choice of estimate is motivated as follows.  Let
$\lambda_i\in\Lambda$ and $\hat\lambda_j\in\hat\Lambda$ be such that
$\mathrm{dist}(\Lambda,\hat\Lambda)=|\lambda_i-\hat\lambda_j|$, and
let $\hat\psi_j\in V$ be the discrete eigenvector associated with
$\hat\lambda_j$; as usual, we assume $(\hat\psi_k,
\hat\psi_\ell)=\delta_{k\ell}$ for $1\leq k,\ell \leq r$.  Let $S$
be the orthogonal projector onto $E(\lambda_i)$.  We have the
well-known identity
\begin{align*}
\enorm{(I-S) \hat\psi_j}^2-\lambda_i \|(I-S) \hat\psi_j\|_0^2=\hat\lambda_j-\lambda_i~.
\end{align*}
If $\hat\lambda_j\geq\lambda_i$, which is certainly the case if $j\geq
i$, then $\mathrm{dist}(\Lambda,\hat\Lambda)\leq \enorm{(I-S)
  \hat\psi_j}^2$.  Note that if the method is converging, then,
asymptotically, we expect $i=j$. 
In any case, we have $\mathrm{dist}(\Lambda,\hat\Lambda)\leq C
\enorm{(I-S) \hat\psi_j}^2$.
Now, 
\begin{align*}
\enorm{(I-S) \hat\psi_j}^2\leq \max_{\substack{v\in \hat{E} \\ \|v\|_0=1}}\enorm{(I-S) v}^2=\max_{\substack{\mb{v}\in\RR^r\\\mb{v}^t\mb{v}=1}}\mb{v}^tH\mb{v}=\lambda_{\max}(H)~.
\end{align*}
Here, we have identified $v\in \hat{E}$ with its coefficient vector
$\mb{v}\in\RR^r$ with respect to the discrete eigenbasis of $\hat{E}$.
Finally, $\lambda_{\max}(H)$ is estimated by
$\lambda_{\max}(\tilde{H})$.
We highlight the difference between
this sort of estimate and the trace-type
estimate~\eqref{TraceEstimates4b},
\begin{align}\label{EigenvalueClusterTraceEst}
\sum_{j=1}^r(\hat\lambda_j-\lambda_j)\approx \mathrm{trace}(\tilde{H})~.
\end{align}
Note that, for the trace-type estimate, we only need computable
approximations of the diagonal entries of $H$, and these may be
obtained using any number of a posteriori techniques for source
problems.  We have opted for the auxiliary subspace approach discussed
in Section~\ref{sec:AuxiliarySubspace}, because it also naturally
provides approximations of the off-diagonal entries of $H$, thereby
permitting estimates of the
form~\eqref{EigenvalueClusterHausdorffEst}.  Using the reference
eigenvalues computed in a rich finite element space ($p=16$), the
errors and error estimates and error estimates for $2\leq p\leq 12$
and some choices of $r$ are given in
Figure~\ref{EigenvalueClusterHausdorffFig}.  Computations were done
for $2\leq r\leq 12$, and the plots shown in
Figure~\ref{EigenvalueClusterHausdorffFig} are representative.  The
effectivities of the error estimate, over all values $2\leq p,r\leq
12$, ranged between 0.574 and 2.469 for Dirichlet Case 2, and between
0.289 and 3.671 for Neumann Case 9.
\begin{figure}
\centering
\subfloat[Dirichlet bridge, Case 2; $r=2$.]
{\label{fig:Haus_BridgeDomainD2_2}
\includegraphics[width=0.45\textwidth]{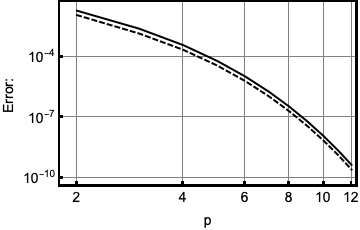}}\quad
\subfloat[Neumann bridge, Case 9; $r=2$.]
{\label{fig:Haus_BridgeDomainN9_2}
\includegraphics[width=0.45\textwidth]{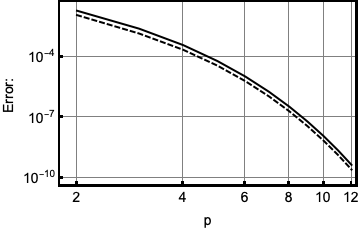}}\\[3pt]
\subfloat[Dirichlet bridge, Case 2; $r=7$.]
{\label{fig:Haus_BridgeDomainD2_7}
\includegraphics[width=0.45\textwidth]{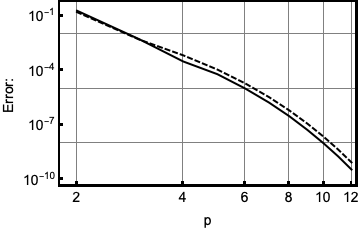}}\quad
\subfloat[Neumann bridge, Case 9; $r=7$.]
{\label{fig:Haus_BridgeDomainN9_7}
\includegraphics[width=0.45\textwidth]{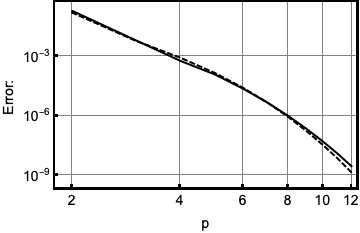}}\\[3pt]
\subfloat[Dirichlet bridge, Case 2; $r=8$.]
{\label{fig:Haus_BridgeDomainD2_8}
\includegraphics[width=0.45\textwidth]{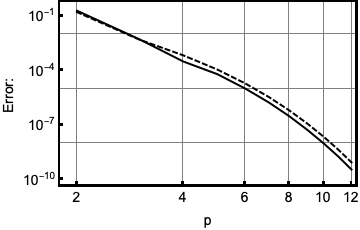}}\quad
\subfloat[Neumann bridge, Case 9; $r=8$.]
{\label{fig:Haus_BridgeDomainN9_8}
\includegraphics[width=0.45\textwidth]{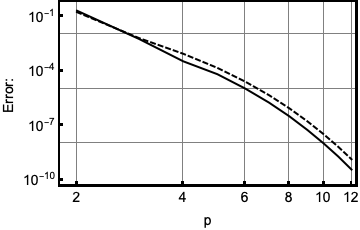}}\\[3pt]
\subfloat[Dirichlet bridge, Case 2; $r=12$.]
{\label{fig:Haus_BridgeDomainD2_12}
\includegraphics[width=0.45\textwidth]{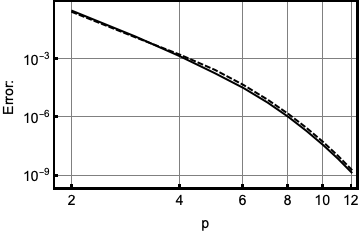}}\quad
\subfloat[Neumann bridge, Case 9; $r=12$.]
{\label{fig:Haus_BridgeDomainN9_12}
\includegraphics[width=0.45\textwidth]{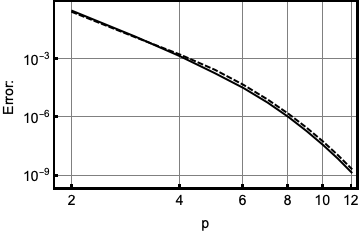}}
\caption{\label{EigenvalueClusterHausdorffFig} Eigenvalue errors
  $\mathrm{dist}(\Lambda,\hat\Lambda)$ (solid) and error estimates
  $\lambda_{\max}(\tilde{H})$ (dashed) for two Bridge domain configurations,
$\Lambda=\{\lambda_1,\ldots,\lambda_r\}$.}
\end{figure}

Letting $E=\mathrm{span}\{\psi_1,\ldots,\psi_r\}$ be the eigenspace of
interest (computed using $p=16$), and
$\hat{E}=\mathrm{span}\{\hat\psi_1,\ldots,\hat\psi_r\}$ be computed
approximations for various discretization parameters $2\leq p\leq 12$,
we consider the subspace gap (cf. Remark~\ref{GapRemark} and~\eqref{SubspaceGap}) and our computable estimate of it,
\begin{align}\label{SubspaceGapHeuristic}
\mathrm{gap}(E,\hat{E})=\sqrt{\lambda_{\max}(G^{-1}H)}\approx \sqrt{\lambda_{\max}(G^{-1}\tilde{H})}~,
\end{align}
where the first equality holds provided $\lambda_{\max}(G^{-1}H)<1$,
as is the case for all of our computations.  Here, we take $S$ to be
the orthogonal projector onto $E$ in the definition of $H$.
The complementary plots to Figure~\ref{EigenvalueClusterHausdorffFig}
for convergence in subspace gap is given in
Figure~\ref{SubspaceGapFig}.  As before, the computable estimates
faithfully reflect the actual subspace gaps, with effectivities
ranging between 0.747 and 0.879
for Dirichlet Case 2, and between 0.527 and 0.874
for Neumann Case 9.  
For comparison, we have also included the trace-type estimate
$\sqrt{\mathrm{trace}(G^{-1}\tilde{H})}$ indicated in~\eqref{TraceEstimates4a} in Figure~\ref{SubspaceGapFig}.
For this estimate, the effectivities ranged between
1.071 and 2.058 for Dirichlet Case 2, and between 1.001 and 2.052 for Neumann
Case 9.
\begin{figure}
\centering
\subfloat[Dirichlet bridge, Case 2; $r=2$.]
{\label{fig:Gap_BridgeDomainD2_2}
\includegraphics[width=0.45\textwidth]{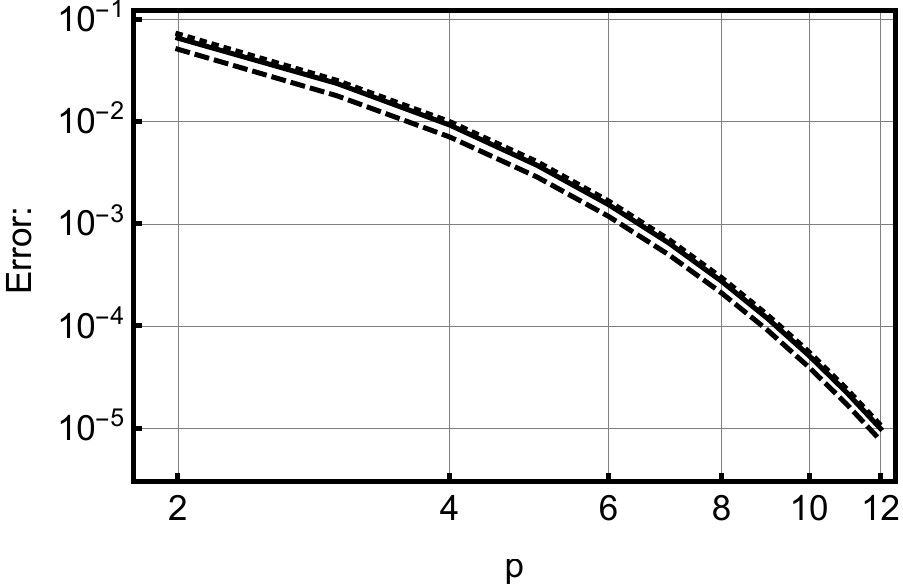}}\quad
\subfloat[Neumann bridge, Case 9; $r=2$.]
{\label{fig:Gap_BridgeDomainN9_2}
\includegraphics[width=0.45\textwidth]{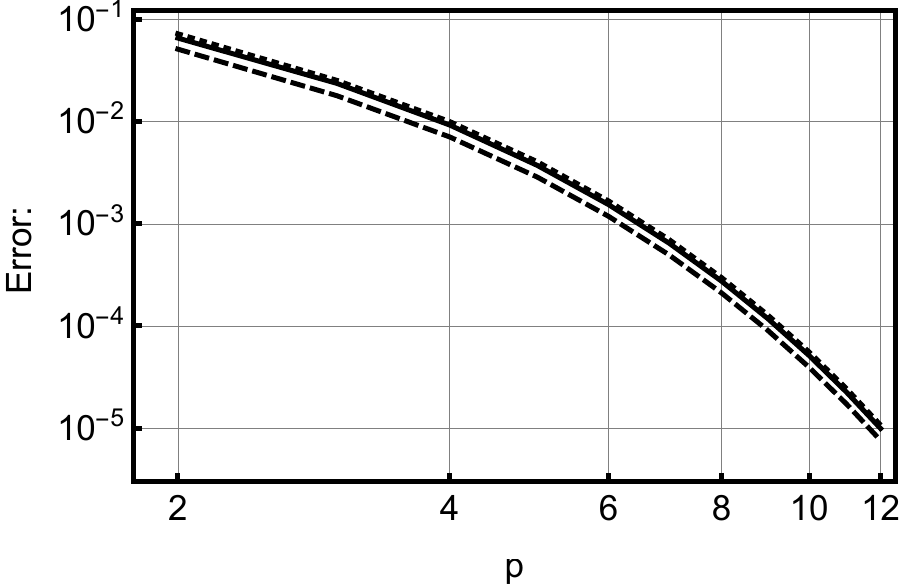}}\\[3pt]
\subfloat[Dirichlet bridge, Case 2; $r=7$.]
{\label{fig:Gap_BridgeDomainD2_7}
\includegraphics[width=0.45\textwidth]{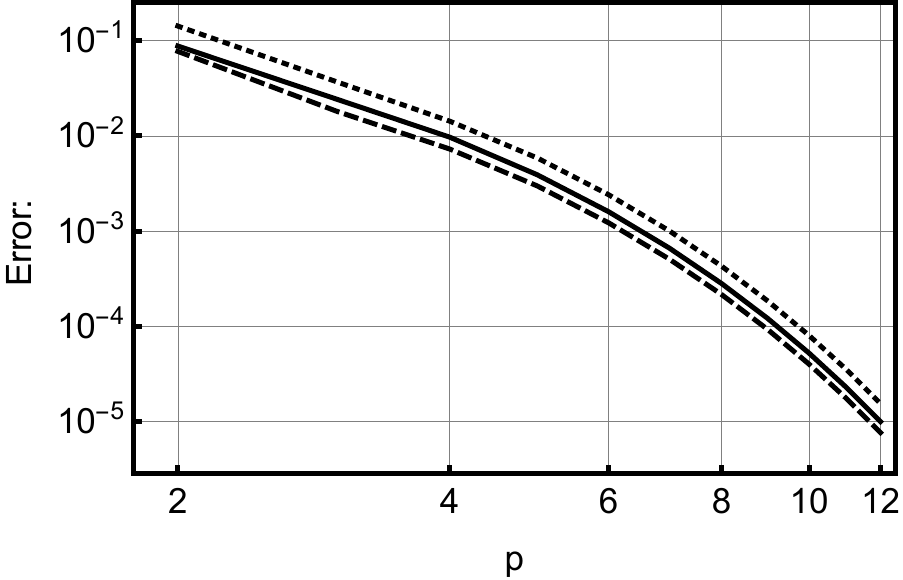}}\quad
\subfloat[Neumann bridge, Case 9; $r=7$.]
{\label{fig:Gap_BridgeDomainN9_7}
\includegraphics[width=0.45\textwidth]{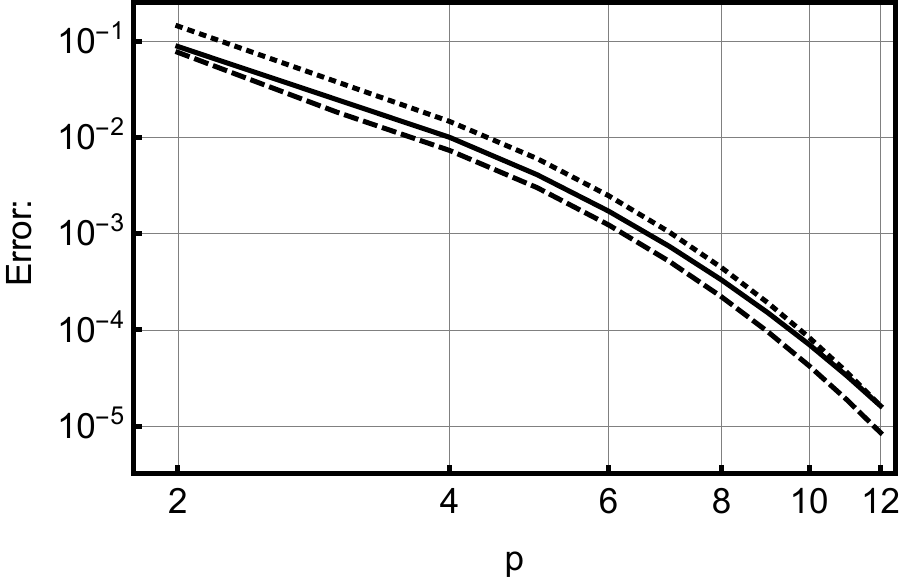}}\\[3pt]
\subfloat[Dirichlet bridge, Case 2; $r=8$.]
{\label{fig:Gap_BridgeDomainD2_8}
\includegraphics[width=0.45\textwidth]{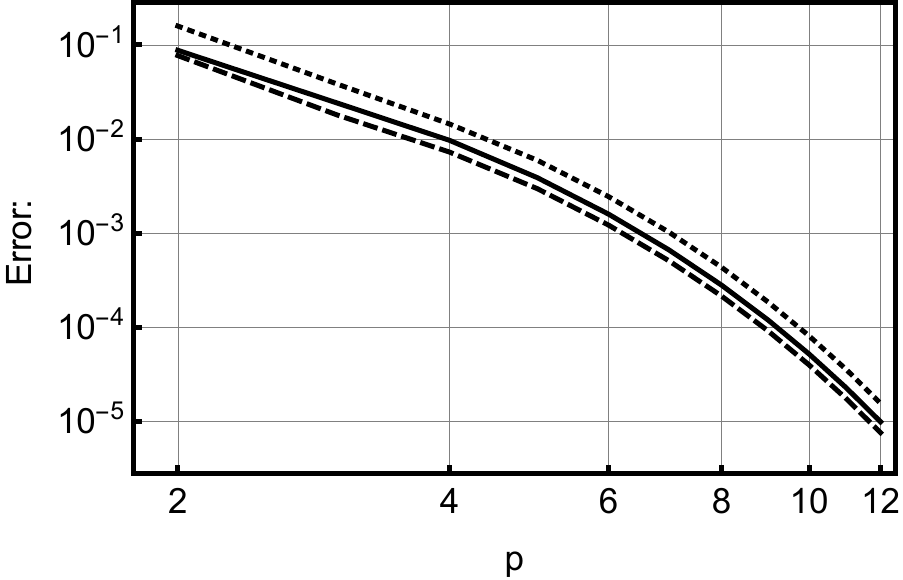}}\quad
\subfloat[Neumann bridge, Case 9; $r=8$.]
{\label{fig:Gap_BridgeDomainN9_8}
\includegraphics[width=0.45\textwidth]{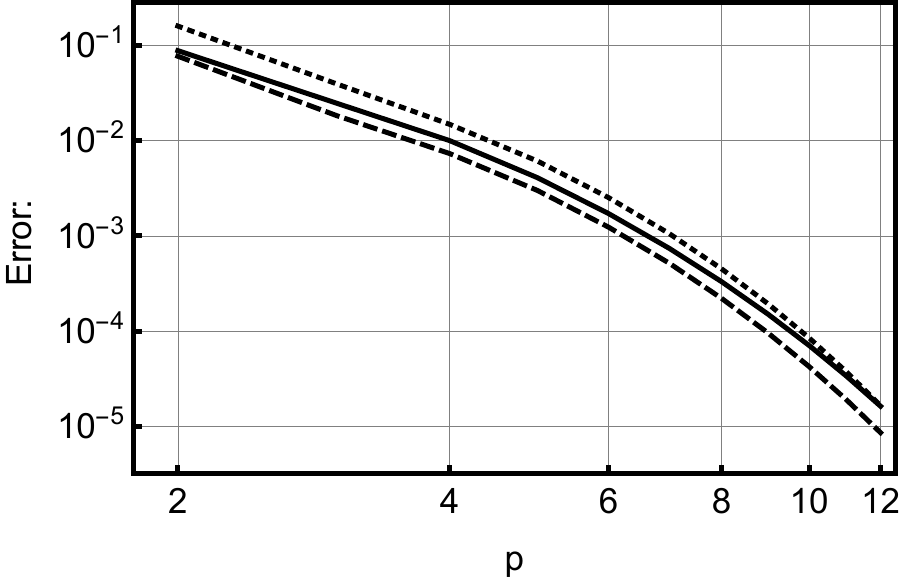}}\\[3pt]
\subfloat[Dirichlet bridge, Case 2; $r=12$.]
{\label{fig:Gap_BridgeDomainD2_12}
\includegraphics[width=0.45\textwidth]{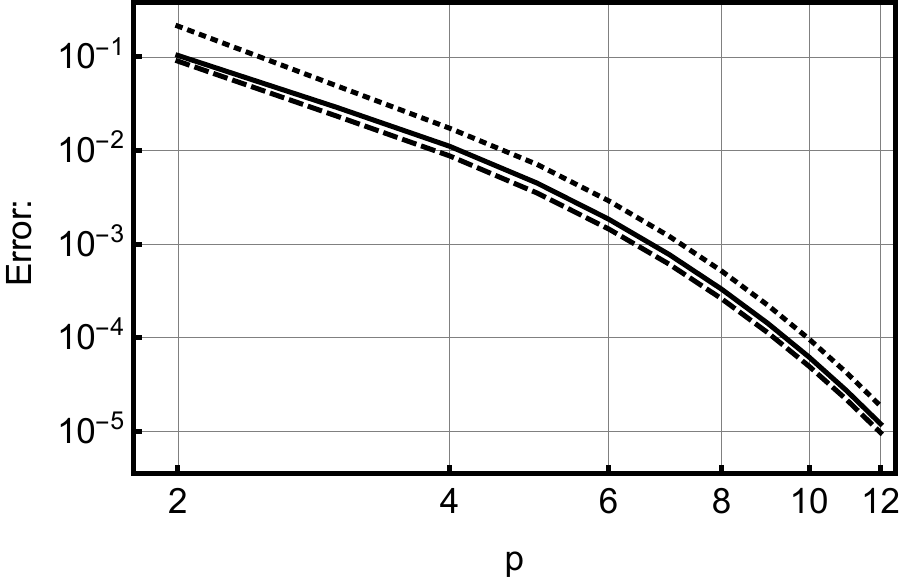}}\quad
\subfloat[Neumann bridge, Case 9; $r=12$.]
{\label{fig:Gap_BridgeDomainN9_12}
\includegraphics[width=0.45\textwidth]{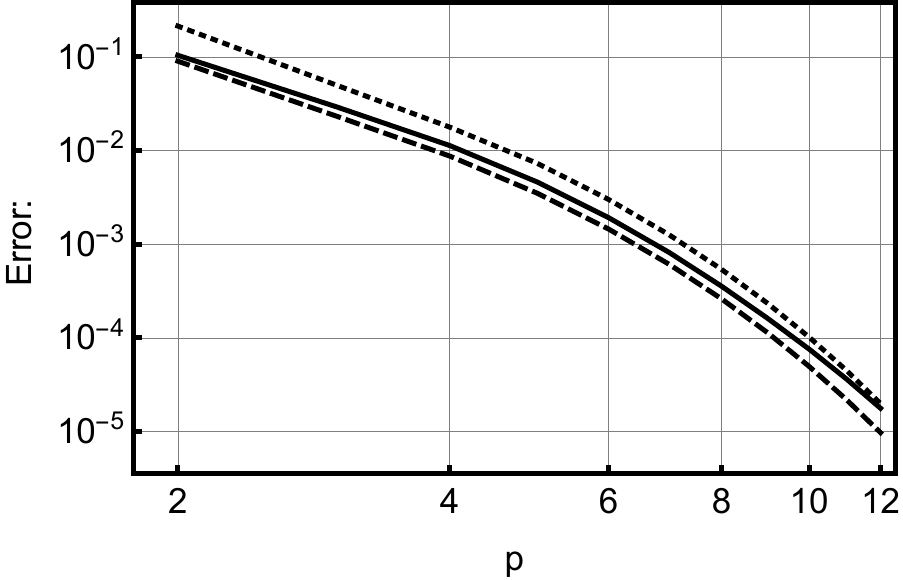}}
\caption{\label{SubspaceGapFig} Subspace gaps $\mathrm{gap}(E,\hat{E})$
(solid) and estimates $\sqrt{\lambda_{\max}(G^{-1}\tilde{H})}$
(dashed) and  $\sqrt{\mathrm{trace}(G^{-1}\tilde{H})}$
(dotted) for two Bridge domain configurations,
$E=\mathrm{span}\{\psi_1,\ldots,\psi_r\}$.}
\end{figure}
\begin{remark}\label{Heuristics}
The estimates in~\eqref{EigenvalueClusterHausdorffEst} and
~\eqref{SubspaceGapHeuristic} employ the heuristics $\lambda_{\max}(\tilde{H})\approx \lambda_{\max}(H)$
and
$\lambda_{\max}(G^{-1}\tilde{H})\approx \lambda_{\max}(G^{-1}H)$.  At
present, we only have emprical evidence that the computable
quantities, i.e. those involving $\tilde{H}$, really do approximate
their typically uncomputable counterparts well.
\end{remark}

\section{Conclusions}\label{sec:Concl}
We have presented computable a posteriori estimates of the subspace
gap between computed and target eigenspaces of the same size, as well
for two measures of error between the corresponding computed and
target eigenvalues---namely, the typical sum of eigenvalue errors and
the Hausdorff distance between the computed and
target eigenvalues.  More rigorous theoretical footing is provided for
the trace-type estimates of the subspace gap~\eqref{TraceEstimates4a}
and sum of eigenvalue errors~\eqref{TraceEstimates4b},
whereas the estimate of the Hausdorff distance between the computed and
target eigenvalues~\eqref{EigenvalueClusterHausdorffEst} and the
alternate estimate of the subspace gap~\eqref{SubspaceGapHeuristic}
is based on the heuristic that the eigenvalues of $H$ and $\tilde{H}$
are close, for which we currently have only empirical support.
These estimates have been tested extensively on a collection of
problems that include both natural clusters of eigenvalues and singularities
in many eigenfunctions.  

\def\cprime{$'$}

\end{document}